\documentclass[leqno,a4paper,10pt]{amsart}
\usepackage{amsmath,amsthm,amssymb,amsfonts,enumerate,color}

\date{}

\newcommand{\norm}[1]{\left\Vert#1\right\Vert}
\newcommand{\abs}[1]{\left\vert#1\right\vert}
\newcommand{\set}[1]{\left\{#1\right\}}

\newtheorem{thm}{Theorem}[section]
\newtheorem{prop}[thm]{Proposition}

\newtheorem{lem}[thm]{Lemma}

\theoremstyle{definition}

\newtheorem{rem}[thm]{Remark}

\numberwithin{equation}{section}
\author[J. L. Torrea]{Jos\'e L. Torrea}
\address{Departamento de Matem\'aticas \\
          Facultad de Ciencias \\
          Universidad Au\-t\'o\-no\-ma de Madrid \\
          28049 Madrid, Spain\\and ICMAT-CSIC-UAM-UCM-UC3M}
\email{joseluis.torrea@uam.es}

\author[C. Zhang]{Chao Zhang}
\address{School of Mathematics and Statistics \\
          Wuhan University \\
          430072 Wuhan, China}
\address{\textit{Current address:}
          \vskip0.01cm Departamento de Matem\'aticas \\
          Facultad de Ciencias \\
          Universidad Au\-t\'o\-no\-ma de Madrid \\
          28049 Madrid, Spain}
\email{zaoyangzhangchao@163.com}

\thanks{Research supported by Ministerio de Ciencia e Innovaci\'{o}n de Espa\~{n}a
MTM2008-06621-C02-01 and the National Natural Science Foundation of
China No.11071190}



\begin{document}

\title{FRACTIONAL VECTOR-VALUED LITTLEWOOD--PALEY--STEIN THEORY FOR SEMIGROUPS}

\begin{abstract}
We consider the fractional derivative of a  general Poisson
semigroup.  With this fractional derivative, we define the
generalized fractional Littlewood--Paley $g$-function for semigroups
acting on $L^p$-spaces of functions with values in Banach spaces. We
give a characterization of  the classes of Banach spaces for which
the fractional Litlewood--Paley $g$-function is bounded on
$L^p$-spaces. We show that the class of Banach spaces is independent
of the order of derivation and coincides with the classical (Lusin
type/cotype) case. It is also shown that the same kind of results
exist for the case of the fractional area function and the
fractional $g^*_\lambda$-function on $\mathbb{R}^n$.

At last, we consider the relationship of the almost sure finiteness
of the fractional Littlewood--Paley $g$-function, area function, and
$g^*_\lambda$-function  with  the  Lusin cotype property of the
underlying Banach space. As a byproduct of the techniques developed,
one can get some results of independent interest for vector-valued
Calder\'on--Zygmund operators. For example, one can get the
following characterization, a Banach space  $\mathbb{B}$ is UMD if
and only if  for some (or, equivalently, for every) $p\in
[1,\infty)$, $\displaystyle \lim_{\varepsilon \rightarrow 0^+ }
\int_{|x-y|> \varepsilon} \frac{f(y)}{x-y}dy $ exists \textup{a.e.}
$x\in \mathbb{R}$ for every
 $f\in L^p_\mathbb{B}(\mathbb{R}).$ \end{abstract}

\maketitle

\section{Introduction}

In the last  decade, a lot of attention has been devoted to the
study of  fractional laplacians, see \cite{cafa, stinga} and the
references therein.  On the other hand, several concepts of
fractional derivatives have been developed in the literature since
19th century. Depending on the motivation  of the researchers, these
two objects can be different and  even unrelated. However, when
dealing with semigroups, it is clear that any definition of
fractional  derivative should have relation with the definition  of
fractional laplacian.   Roughly speaking, a fractional derivative
(with respect to $t$) of order ``$\alpha$'' of the Poisson
semigroup, $e^{-t\sqrt{\mathcal{L}}}$,  of a certain differential
operator $\mathcal{L}$, should be closely related to
$\mathcal{L}^{\alpha/2}e^{-t\sqrt{\mathcal{L}}}.$

 Segovia and Wheeden, see \cite{SeWh}, motivated by some characterization
of potential spaces on $\mathbb{R}^n,$ introduced the following
definition of ``fractional derivative" $\partial^\alpha.$ Given
$\alpha>0,$ let $m$ be the smallest integer which strictly exceeds
$\alpha.$ Let $f$ be a reasonable nice function in
$L_{\mathbb{B}}^p\big(\mathbb{R}^n \big).$ Then
\begin{equation*}
\partial_t^{\alpha}\mathcal{P}_tf(x)=\frac{e^{-i \pi(m-\alpha)}}{\mathbf{\Gamma}(m-\alpha)}\int_0^\infty
{\partial_t^m}\mathcal{P}_{t+s}(f)(x)s^{m-\alpha-1}ds, \qquad t>0,  x \in \mathbb{R}^n,
\end{equation*}
where $\mathbf{\Gamma}$ denotes the Gamma function and
$\mathcal{P}_t$ denotes the classical Poisson semigroup on
$\mathbb{R}^n$ . Observe that for reasonable good functions,
$\partial_t^{\alpha}\mathcal{P}_tf(x) = e^{i\pi\alpha}
(-\Delta)^{\alpha/2} \mathcal{P}_tf(x).$ In \cite{SeWh}, the authors
developed a satisfactory theory of euclidean square functions of
Littlewood--Paley type in which the usual derivatives are
substituted by these fractional derivatives.

It turns out that the notion of partial derivative considered by
Segovia and Wheeden can be used  in the case of general subordinated
Poisson semigroups defined on a measure space $(\Omega, d\mu),$ see
Section \ref{prepare}. Of course, without having a pointwise
expression but just an identity in $L^p(\Omega)$. This fractional
derivative has a nice behavior for iteration and for spectral
decomposition.  Then it is natural to ask whether  results already
known for classical derivatives are still true for the fractional
derivative case.  In this paper, we  shall be concerned with several
characterizations of Lusin type and Lusin cotype of Banach spaces by
the boundedness of square functions  defined by using the fractional
derivatives.  Now we explain briefly the concept of Lusin type and
Lusin cotype.

The martingale type and cotype properties of a Banach space
$\mathbb{B}$ were introduced in the 1970's by G. Pisier, see
\cite{Pis1,Pis2}, in connection with the convexity and smoothness of
the Banach space $\mathbb{B}$. If $M=(M_n)_{n \in \mathbb{N}}$ is a
martingale defined on some probability space and with values in
$\mathbb{B}$, the $q$-square function $S_qM$ is defined by
$\displaystyle S_q
M=\Big(\sum^{\infty}_{n=1}\|M_n-M_{n-1}\|^q_{\mathbb{B}}\Big)^{\frac{1}{q}}.
$ The Banach space $\mathbb{B}$ is said to be of martingale cotype
$q$, $2\leq q <\infty$, if  for every bounded
$L_{\mathbb{B}}^p$-martingale $M=(M_n)_{n\in \mathbb{N}}$ we have
$\displaystyle \left\|S_qM \right\|_{L^p} \leq C_p
\sup_n\|M_n\|_{L^p_{\mathbb{B}}}, $ for some $1<p<\infty$. The
Banach space $\mathbb{B}$ is said to be of martingale type $q$, $1 <
q\le 2,$ when  the reverse inequality holds
 for some $1<p<\infty$. The martingale type and cotype
properties do not depend on $1<p<\infty$ for which the corresponding
inequalities are satisfied. $\mathbb{B}$ is of martingale cotype $q$
if and only if its dual, $\mathbb{B}^*$,  is of martingale type $q'=
q/(q-1).$

It is a common fact that results in probability theory have
parallels in harmonic analysis. In this line of thought, Xu, see
\cite{Xu1}, defined  the Lusin cotype and Lusin type properties for
a Banach space $\mathbb{B}$ as follows.  Let  $f$ be a function in
$L^1(\mathbb{T},\mathbb{B})$, where $\mathbb{T}$ denotes the one
dimensional torus and $L^1(\mathbb{T},\mathbb{B})$ stands for the
Bochner--Lebesgue space of strong measurable $\mathbb{B}$-valued
functions such that the scalar function $\|f\|_{\mathbb{B}}$ is
integrable. Consider the generalized Littlewood--Paley $g$-function
$$
g_q(f)(z)= \left(\int_0^1 (1-r)^q\left\|{ \partial_rP_r} \ast
f(z)\right\|_{\mathbb{B}}^q\frac{dr}{1-r}\right)^{\frac{1}{q}},
$$
where $P_r(\theta)$ denotes  the Poisson kernel. It is said that
$\mathbb{B}$ is of Lusin cotype $q,$  $q \geq 2,$ if for some $1<p
<\infty$ we have $\displaystyle  \left\|g_q(f)
\right\|_{L^p(\mathbb{T})} \leq C_p
\|f\|_{L^p_{\mathbb{B}}(\mathbb{T})}, $ and $\mathbb{B}$ is of Lusin
type $q,$  $1 \leq q \leq 2,$ if for some $1<p<\infty$ we have
$\displaystyle \|f\|_{L^p_{\mathbb{B}}(\mathbb{T})} \leq C_p
\left(\| \hat{f}(0) \|_{\mathbb{B}} + \left\|g_q(f)
\right\|_{L^p(\mathbb{T})} \right). $

The Lusin cotype and Lusin type properties do not depend on $p \in
(1,\infty)$, see \cite{Xu1,OX}. Moreover, a Banach space
$\mathbb{B}$ is of Lusin cotype $q$ (Lusin type $q$) if and only if
$\mathbb{B}$ is of martingale cotype $q$ (martingale type $q$), see
\cite[Theorem 3.1]{Xu1}.

 Mart\'{\i}nez, Torrea and Xu, see \cite{MTX}, extended the results in
\cite{Xu1} to subordinated Poisson semigroup
$\set{\mathcal{P}_t}_{t\ge 0}$ of a general symmetric diffusion
Markovian semigroup $\{\mathcal{T}_t\}_{t\ge 0}$.  That is,  a
family of linear operators defined on $L^p(\Omega, d\mu)$ over a
measure space $(\Omega, d\mu)$ satisfying the semigroup properties
\begin{itemize}
\item $\mathcal
{T}_0=\text{Id}, \mathcal{T}_t\mathcal{T}_s=\mathcal{T}_{t+s}.$
\item $\left\|\mathcal{T}_t\right\|_{L^p\rightarrow L^p}\leq 1 \quad
\forall p\in [1, \infty],$
\item $
 \lim\limits_{t\rightarrow 0}\mathcal{T}_tf=f \quad
\text{in} \quad L^2\quad \forall f\in L^2,$
\item $
\mathcal{T}_t^*=\mathcal{T}_t \quad \text{on} \quad L^2 \ \textup{
and}$
\item $
 \mathcal{T}_t f\geq 0 \quad \text{if}\  f\geq 0, \quad
\mathcal{T}_t 1=1.$
\end{itemize}
The subordinated Poisson semigroup $\{\mathcal{P}_t\}_{t\geq 0}$ (again a symmetric diffusion semigroup, see \cite{Stein1}) is
defined as
\begin{equation}
\label{1.4} \mathcal{P}_tf =\frac{t}{2\sqrt\pi}\int_0^\infty \frac{~e^{-\frac{t^2}{4u}}}{
u^{\frac{3}{2}}}\mathcal{T}_{u}f du.
\end{equation}
 Being positive operators,  $\mathcal{T}_t$
and $\mathcal{P}_t$ have straightforward norm-preserving extensions to
$L_{\mathbb{B}}^p(\Omega)$ for every Banach space $\mathbb{B}$,
where $L_{\mathbb{B}}^p(\Omega)$ denotes the usual Bochner--Lebesgue
$L^p$-space of $\mathbb{B}$-valued functions defined on $\Omega.$
 Let $g_1^q$ be the generalized Littlewood--Paley $g$-function defined by
 \begin{equation*}\label{0.0}
g_1^q(f)(x)=\left(\int_0^\infty \left\|t
\partial_t \mathcal
{P}_tf(x)\right\|_\mathbb{B}^q\frac{dt}{t}\right)^{\frac{1}{q}},\qquad
\forall f\in \bigcup\limits_{1\leq p\leq
\infty}L_\mathbb{B}^p(\Omega).
\end{equation*} The results in \cite{MTX} are as follows.

\begin{thm}\label{cotipo}
Given a Banach space $\mathbb{B}$  and $2\le q < \infty,$ the following statements are equivalent:
\begin{itemize}
\item [(i)]  $\mathbb{B}$ is of Lusin cotype $q.$
\item [(ii)] For every subordinated Poisson semigroup $\{\mathcal{P}_t\}_{t\ge 0}$
and for every (or, equivalently, for some) $p\in (1,\infty)$, there
is a constant $C$
 such that $\displaystyle \| g_1^q(f) \|_{ L^p(\Omega)} \le C
 \| f \|_{ L^p_{\mathbb{B}}(\Omega)} , $ for every $f \in L^p_{\mathbb{B}}(\Omega)$.
 \end{itemize}
\end{thm}
 Let $\mathbb{E}_0\subset L^2(\Omega)$
be the subspace of all $h$ such that $\mathcal{P}_t(h)=h$ for all
$t\geq 0.$ Let $E_0: L^2( \Omega)\longrightarrow \mathbb{E}_0$ be
the orthogonal projection. $E_0$ extends to a contractive projection
on $L^p(\Omega)$ for every $1\leq p<\infty$.  $E_0(L^p(\Omega))$ is
exactly the fix point space of $\{\mathcal{P}_t\}_{t\geq 0}$ on
$L^p(\Omega),$ see \cite{Stein1}. Moreover, for any Banach space
$\mathbb{B}$, $E_0$ extends to a contractive projection on
$L_\mathbb{B}^p(\Omega)$ for every $1\leq p< \infty$ and
$E_0(L_\mathbb{B}^p(\Omega))$ is again the fix point space of
$\{\mathcal{P}_t\}_{t\geq 0}$ considered as a semigroup on
$L_\mathbb{B}^p(\Omega)$. In the particular case on $\mathbb{R}^n,$
$\mathbb{E}_0=0$ and so $E_0(L_\mathbb{B}^p(\mathbb{R}^n))=\{0\}$.
 In the sequel, we shall use the same
symbol $E_0$ to denote any of these contractive projections.

\begin{thm}\label{tipo}
Given a Banach space $\mathbb{B}$  and $1< q \le 2,$ the following statements are equivalent:
\begin{itemize}
\item [(i)]  $\mathbb{B}$ is of Lusin type $q.$
\item [(ii)] For every subordinated Poisson semigroup $\{\mathcal{P}_t\}_{t\ge 0}$
and for every (or, equivalently, for some) $p\in (1,\infty)$, there
is a constant $C$
 such that $\displaystyle   \| f \|_{ L^p_{\mathbb{B}}(\Omega)} \le C \Big(  \| E_0(f)
 \|_{ L^p_{\mathbb{B}}(\Omega)} +  \| g_1^q(f) \|_{ L^p(\Omega)}  \Big),$
for every $f \in L^p_{\mathbb{B}}(\Omega)$.
 \end{itemize}
\end{thm}

\begin{thm}\label{tipoa.e.}
Given a Banach space $\mathbb{B}$  and $2\le q<\infty,$ the
following statements are equivalent:
\begin{itemize}
\item [(i)]  $\mathbb{B}$ is of Lusin cotype $q.$
\item [(ii)] For any $f \in L^1_{\mathbb{B}}(\mathbb{T}),$ $g_q(f)(z)<\infty$ for almost every $z\in \mathbb{T}.$
\item [(iii)] For any $f \in L^1_{\mathbb{B}}(\mathbb{R}^n),$ $g_1^q(f)(x)<\infty$ for almost every $x\in \mathbb{R}^n.$
 \end{itemize}
\end{thm}

As we said before, our goal is to characterize Lusin cotype and
Lusin type properties of Banach spaces when the standard derivative
is substitute by the fractional derivative. Parallel  to Segovia and
Wheeden, we define
\begin{equation}
\label{2.1}
\partial_t^{\alpha}\mathcal{P}_tf=\frac{e^{-i \pi(m-\alpha)}}{\mathbf{\Gamma}(m-\alpha)}\int_0^\infty
{\partial_t^m}\mathcal{P}_{t+s}(f)s^{m-\alpha-1}ds, \qquad t>0,
\end{equation}
where $m$ is the smallest integer which strictly exceeds $\alpha.$
In Section \ref{prepare}, we shall see  that for any $f \in
L^p(\Omega)$, this partial derivative is well defined  and then we
are allowed to consider the following ``fractional Littlewood--Paley
$g$-function"
\begin{equation}\label{1.7}
g_\alpha^q(f)=\left(\int_0^\infty \left\|t^\alpha
\partial_t^\alpha \mathcal
{P}_tf\right\|_\mathbb{B}^q\frac{dt}{t}\right)^{\frac{1}{q}},\qquad
 f\in \bigcup\limits_{1\leq p\leq
\infty}L_\mathbb{B}^p(\Omega), \, \alpha >0.
\end{equation}
The results in this paper can be classified in three types:
\begin{itemize}
\item Theorems which generalize the results in \cite{MTX} for the case of fractional derivatives (Theorem A and Theorem
B).
\item New theorems involving area functions and  $g^*_\lambda$ functions on $\mathbb{R}^n$ (Theorem \ref{thm4} -- Theorem
\ref{thm9}).
\item New results for characterizations of Lusin cotype through almost everywhere finiteness (Theorem
C).
\end{itemize}

In our opinion, it is worth to mention that the proof of Theorem C
contains some new ideas that can be applied to a huge class of
operators. Roughly, the method used in the proof is  the following.
If an operator $T$ with  a Calder\'on--Zygmund kernel is
a.e.~pointwise finite  ($Tf(x) < \infty$) for any function $f$ in
$L^{p_0}(\mathbb{R}^n)$ and some $p_0\in [1, \infty)$, then $T$ is
bounded from $L^1(\mathbb{R}^n)$ into weak-$L^1(\mathbb{R}^n)$. This
philosophy can be translated to the vector-valued case and we can
get results like the one presented in Theorem D.

Now we list our main theorems.

\

\noindent \textbf{Theorem A.}  Given a Banach space $\mathbb{B}$ and
$2\le q < \infty $, the following statements are equivalent:
\begin{enumerate}
\item[(i)] $\mathbb{B}$ is of Lusin cotype $q$.
\item[(ii)] For every symmetric diffusion semigroup $\{\mathcal{T}_t\}
_{t\geq 0}$ with subordinated semigroup $\{\mathcal {P}_t\}_{t\geq
0}$,  for every (or, equivalently, for some) $p\in (1,\infty)$, and
for every (or, equivalently, for some) $\alpha >0$, there is a
constant $C$ such that
\begin{equation*}
\left\|g_{\alpha}^q(f)\right\|_{L^p(\Omega)}\leq
C\|f\|_{L_\mathbb{B}^p(\Omega)}, \quad \forall f\in
L_\mathbb{B}^p(\Omega).
\end{equation*}
\end{enumerate}

\

\noindent \textbf{Theorem B.} Given a Banach space $\mathbb{B}$ and
$1<q\leq 2$ , the following statements are
equivalent:
\begin{enumerate}
\item[(i)] $\mathbb{B}$ is of Lusin  type $q$.
\item[(ii)] For every symmetric diffusion semigroup $\{\mathcal {T}_t\}_{t\geq 0}$ with subordinated semigroup
$\{\mathcal {P}_t\}_{t\geq 0}$, for every (or, equivalently, for
some) $p\in (1,\infty)$, and for every (or, equivalently, for some)
$\alpha >0$, there is a constant $C$ such that
$$
\left\|f\right\|_{L_\mathbb{B}^p(\Omega)}\leq
C\left(\left\|E_0(f)\right\|_{L_\mathbb{B}^p(\Omega)}+
\left\|g_{\alpha}^q(f)\right\|_{L^p(\Omega)}\right),\qquad \forall
f\in L_\mathbb{B}^p(\Omega).$$
\end{enumerate}
\

On the particular Lebesgue measure space
$\left(\mathbb{R}^n,dx\right),$ we have the following theorems.

\

\noindent \textbf{Theorem C.} Given a Banach space $\mathbb{B}$,
$2\leq q<\infty$, the following statements are equivalent:
\begin{enumerate}
\item[(i)] $\mathbb{B}$ is of Lusin cotype $q.$
\item[(ii)] For every (or, equivalently, for
some) $p\in [1,\infty)$ and for every (or, equivalently, for some)
$\alpha >0,$ $g_\alpha^q(f)(x)<\infty$ for a.e. $x\in\mathbb{R}^n,$
for every $f\in L_\mathbb{B}^p(\mathbb{R}^n).$
\item[(iii)] For every (or, equivalently, for
some) $p\in [1,\infty)$ and for every (or, equivalently, for some)
$\alpha >0,$ $S_\alpha^q(f)(x)<\infty$ for a.e. $x\in\mathbb{R}^n,$
for every $f\in L_\mathbb{B}^p(\mathbb{R}^n).$
\item[(iv)] For every (or, equivalently, for
some) $p\in [q,\infty)$ and for every (or, equivalently, for some)
$\alpha >0,$
 $g_{\lambda, \alpha}^{q, *}(f)(x)<\infty$ for a.e. $x\in\mathbb{R}^n,$ for every
$f\in L_\mathbb{B}^p(\mathbb{R}^n).$
\end{enumerate}

\

\noindent \textbf{Theorem D.} Given a Banach space $\mathbb{B}$,
  the following statements are
equivalent:
\begin{enumerate}
\item[(i)] $\mathbb{B}$ is UMD.
\item[(ii)]  For every (or, equivalently, for some) $p\in [1,\infty)$, $$ \lim_{\varepsilon \rightarrow 0^+ }
 \int_{|x-y|> \varepsilon} \frac{f(y)}{x-y}dy\, \,  \hbox{exists \textup{a.e.} } x\in \mathbb{R}, \, \,  \hbox{ for every} \, \, \,
 f\in L^p_\mathbb{B}(\mathbb{R}).$$
\end{enumerate}

We want to mention  that in \cite{H} T. Hyt\"onen extended some
results in \cite{MTX}  to the case of  appropriated stochastic
integrals. On the other hand we think that this paper contains some
new conical square function estimates in the sense of \cite{HNP}.

The paper is organized as follows. In Section \ref{prepare}, we
present a systematic study of some properties related to the
fractional derivatives for general Poisson semigroups. Section
\ref{tech g-function} is devoted to the analysis of several
relations between the fractional Litlewood--Paley $g$-functions,
some of them for general Poisson semigroups, while others are for
Poisson semigroups on $\mathbb{R}^n$.  In this case, we need the
theory of Calder\'on--Zygmund as a fundamental tool. Section
\ref{g-function} contains the proofs of Theorem A and Theorem B.
Section \ref{onRn} is devoted to discuss the similar results for the
fractional area function and the fractional $g_\lambda^*$-function
on $\mathbb{R}^n.$ Section \ref{A.E.} is devoted to the proof of
Theorem C. Finally we prove Theorem D in Section \ref{UMD}.

Throughout this paper, the letter $C$ will denote a positive
constant which may change from one instance to another and depend on
the parameters involved. We will make a frequent use, without
mentioning it in relevant places, of the fact that for a positive
$A$ and a non-negative $a,$
$$\sup\limits_{t>0}t^a
\exp(-At)=C_{a, A}<\infty.$$

\section{Fractional Derivatives}\label{prepare}

In this section, we shall consider the general symmetric diffusion
semigroup $\{\mathcal{T}_t\}_{t\geq 0}$ defined on $L^p(\Omega).$
Given such a semigroup $\{\mathcal{T}_t\}_{t\geq 0}$, we consider
its subordinated semigroup $\{\mathcal{P}_t\}_{t\geq 0}$ defined as
in \eqref{1.4}.

\begin{thm}\label{thm2.1}
Given a Banach space $\mathbb{B},$ $1\leq p\leq \infty,$
$\alpha>0,$ and $t>0,$ $\partial_t^\alpha\mathcal{P}_tf$ is well
defined as a function in $L_{\mathbb{B}}^p\left(\Omega \right)$ for
any $f\in L_{\mathbb{B}}^p\left(\Omega \right).$ Moreover, there
exists a constant $C_\alpha$ such that
\begin{equation}\label{02.3}
\left\|\partial_t^\alpha\mathcal{P}_tf\right\|_{L_{\mathbb{B}}^p\left(\Omega\right)}
\leq\frac{C_\alpha}{t^\alpha}\left\|f\right\|_{L_{\mathbb{B}}^p\left(\Omega\right)},
\quad \forall f\in L_{\mathbb{B}}^p\left(\Omega \right).
\end{equation}
\end{thm}
\begin{proof}
Firstly, let us consider the case $\alpha=m,\ m=1,2,\ldots.$ We know
that, for any $m=1,2,\ldots,$ there exist  constants $C_m$ such that
\begin{equation*}
\partial_t^m\left(\frac{t}{\sqrt{u}}~e^{-{\frac{t^2}{4u}}}\right)\leq
C_m\frac{1}{\left(\sqrt{u}\right)^m}~e^{-{\frac{t^2}{4u}}}.
\end{equation*}
Then, by using formula (\ref{1.4}), we have
\begin{multline}\label{02.2}
\left\|\partial_t^m\mathcal{P}_tf\right\|_{L_{\mathbb{B}}^p\left(\Omega\right)}
  \leq C\int_0^\infty \left|\partial_t^m\left(\frac{t}{\sqrt{u}}~e^{-{\frac{t^2}{4u}}}\right)\right|
        \left\|\mathcal{T}_{u}f\right\|_{L_{\mathbb{B}}^p\left(\Omega\right)}\frac{du}{u}\\
  \leq C_m\int_0^\infty \frac{1}{\left(\sqrt{u}\right)^m}~e^{-{\frac{t^2}{4u}}}\frac{du}{u}
         \left\|f\right\|_{L_{\mathbb{B}}^p\left(\Omega\right)}
  = \frac{C_m}{t^m}\left\|f\right\|_{L_{\mathbb{B}}^p\left(\Omega\right)}.
\end{multline}
So we have proved \eqref{02.3} when $\alpha$ is integer.  Therefore, given $\alpha >0$, we have
\begin{align}\label{beta}
\left\|\partial_t^\alpha\mathcal{P}_tf\right\|_{L_{\mathbb{B}}^p\left(\Omega\right)}
  &=   \left\|\frac{e^{-i\pi (m-\alpha)}}{\mathbf{\Gamma}(m-\alpha)}\int_0^\infty\partial_t^m\mathcal{P}_{t+s}(f)s^{m-\alpha-1}ds
        \right\|_{L_{\mathbb{B}}^p\left(\Omega\right)}\nonumber\\
 &\leq \frac{C_m}{\mathbf{\Gamma}(m-\alpha)}\left\|f\right\|_{{L_{\mathbb{B}}^p
        \left(\Omega\right)}} \int_0^\infty\frac{1}{(t+s)^{m}}
        s^{m-\alpha-1}{ds}\\
  &=    \frac{C_m}{\mathbf{\Gamma}(m-\alpha)} \mathbf{B}(m-\alpha, \alpha) \frac{  \left\|f\right\|_{{L_{\mathbb{B}}^p
        \left(\Omega\right)}}}{ t^\alpha }
   =   C_\alpha    \frac{  \left\|f\right\|_{{L_{\mathbb{B}}^p
        \left(\Omega\right)}}}{ t^\alpha } ,\nonumber
 \end{align}
where $\mathbf{B}$ denotes the Beta function, see \cite{Leb}.
\end{proof}

Observe that by estimate (\ref{02.2}),  we can perform  integration
by parts in  the formula (\ref{2.1}). In particular, the formula
\eqref{2.1} is valid for $\alpha$ being integer.

\begin{thm}\label{thm2.2}
Given a Banach space $\mathbb{B}$ and $0<\beta<\gamma,$ we have
\begin{equation}\label{02.5}
\partial_t^\beta\mathcal{P}_{t}f
=\frac{e^{-i\pi(\gamma-\beta)}}{\mathbf{\Gamma}(\gamma-\beta)}\int_0^\infty\partial_t^\gamma
\mathcal{P}_{t+s}(f)s^{\gamma-\beta-1}ds,\quad \forall f\in
\bigcup\limits_{1\leq p\leq \infty}L_{\mathbb{B}}^p
\left(\Omega\right).
\end{equation}
\end{thm}
\begin{proof}
Assume that $f\in L_{\mathbb{B}}^p \left(\Omega\right)$ for some
$1\leq p\leq \infty,$ by changing variables and Fubini's theorem, we
have the following computation as in \eqref{beta}
\begin{align}\label{02.7}
\int_0^\infty
\partial_t^\gamma\mathcal{P}_{t+s}(f)s^{\gamma-\beta-1}ds
   &=\int_0^\infty\frac{e^{-i\pi(k-\gamma)}}{\mathbf{\Gamma}(k-\gamma)}\int_0^\infty \partial_t^k
        \mathcal{P}_{t+s+u}(f)u^{k-\gamma-1}dus^{\gamma-\beta-1}ds\nonumber\\
   &=\frac{e^{-i\pi(k-\gamma)}}{\mathbf{\Gamma}(k-\gamma)}\int_0^\infty \int_s^\infty \partial_t^k
        \mathcal{P}_{t+\bar{u}}(f)(\bar{u}-s)^{k-\gamma-1}s^{\gamma-\beta-1}d\bar{u}ds\\
   &= \frac{e^{-i\pi(k-\gamma)}\mathbf{B}(k-\gamma, \gamma-\beta)}{\mathbf{\Gamma}(k-\gamma)}
        \int_0^\infty  \partial_t^k
        \mathcal{P}_{t+\bar{u}}(f)\bar{u}^{k-\beta-1}d\bar{u},\nonumber
\end{align}
where $k$ is the smallest integer which is bigger than $\gamma.$ By
\eqref{02.2}, we know that we can integrate by parts in the last
integral of \eqref{02.7}. Let $m$ be the smallest integer which is
bigger than $\beta$. Then by integrating by parts $k-m$ times, we
obtain
\begin{align*}
 & \int_0^\infty\partial_t^\gamma
\mathcal{P}_{t+s}(f)s^{\gamma-\beta-1}ds\nonumber\\
&= \frac{\mathbf{B}(k-\gamma,
\gamma-\beta)e^{-i\pi(m-\gamma)}}{\mathbf{\Gamma}(k-\gamma)}
      (k-\beta-1)\cdots(m-\beta)
      \int_0^\infty\partial_t^{m} \mathcal{P}_{t+\bar{u}}(f)
      \bar{u}^{m-\beta-1}d\bar{u}\nonumber\\
&=e^{-i\pi(\gamma-\beta)}\mathbf{\Gamma}(\gamma-\beta)
    \partial_t^\beta \mathcal{P}_{t}f.\nonumber
\end{align*}
Hence we get \eqref{02.5}.
\end{proof}

\begin{thm}\label{thm2.3}
Given a Banach space $\mathbb{B}$ and $\alpha,\ \beta>0,$
$\partial_t^\alpha\left(\partial_t^\beta\mathcal{P}_{t}f\right)$ can
be defined as
\begin{equation}\label{02.8}
\partial_t^\alpha\left(\partial_t^\beta\mathcal{P}_{t}f\right)
   = \frac{e^{-i\pi(m-\alpha)}}{\mathbf{\Gamma}(m-\alpha)}\int_0^\infty{\partial_t^m}
    \left(\partial_{t+s}^\beta\mathcal{P}_{t+s}f\right)s^{m-\alpha-1}ds,
     \quad \forall f\in\bigcup\limits_{1\leq p\leq \infty}L_{\mathbb{B}}^p
\left(\Omega\right),
\end{equation}
where $m$ is the smallest integer which is bigger than $\alpha$.
Then \begin{equation}\label{02.6}
\partial_t^\alpha\left(\partial_t^\beta\mathcal{P}_{t}f\right)
=\partial_t^{\alpha+\beta}\mathcal{P}_{t}f, \quad \forall f\in
\bigcup\limits_{1\leq p\leq \infty}L_{\mathbb{B}}^p
\left(\Omega\right).
\end{equation}
\end{thm}
\begin{proof}
For any $f\in L_{\mathbb{B}}^p \left(\Omega\right)$ for some $1\le
p\le \infty,$ by \eqref{2.1} and Theorem \ref{thm2.1} we have the
following computation for the latter of \eqref{02.8}:
\begin{multline}\label{02.9}
\frac{e^{-i\pi(m-\alpha)}}{\mathbf{\Gamma}(m-\alpha)}\int_0^\infty{\partial_t^m}
    \left(\partial_{t+s}^\beta\mathcal{P}_{t+s}f\right)s^{m-\alpha-1}ds\\
 =\frac{e^{-i\pi(m+k-\alpha-\beta)}}{\mathbf{\Gamma}(m-\alpha)\mathbf{\Gamma}(k-\beta)}\int_0^\infty{\partial_t^m}
    \left(\int_0^\infty\partial_{t+s}^k
    \mathcal{P}_{t+s+u}(f)u^{k-\beta-1}du\right)s^{m-\alpha-1}ds,
\end{multline}
where $k$ is the smallest integer which is bigger than $\beta$. For
any fixed $s \in (0,\infty),$  $t\in (t_0-\varepsilon,
t_0+\varepsilon)\subset (0, \infty)$ for some $t_0\in (0, \infty),$
and $\varepsilon>0$, by \eqref{02.3} we have
\begin{multline}\label{02.11}
\left\|\partial_t^m\left({\partial_{t+s}^k}
        \mathcal{P}_{t+s+u}(f)u^{k-\beta-1}\right)\right\|_{L_{\mathbb{B}}^p(\Omega)}
   =\left\|\partial_t^{m+k} \mathcal{P}_{t+s+u}(f)\right\|_{L_{\mathbb{B}}^p(\Omega)}u^{k-\beta-1}\\
   \leq\frac{C}{(t+s+u)^{m+k}}u^{k-\beta-1}\left\|f\right\|_{L_{\mathbb{B}}^p(\Omega)}
   \leq\frac{C}{(t_0-\varepsilon+s+u)^{m+k}}u^{k-\beta-1}\left\|f\right\|_{L_{\mathbb{B}}^p(\Omega)},
\end{multline}
for any $1\leq p\leq \infty.$ And
\begin{align}\label{02.10}
 &\int_0^\infty\abs{ \frac{u^{k-\beta-1}}{(t_0-\varepsilon+s+u)^{m+k}}}du\left\|
       f\right\|_{L_{\mathbb{B}}^p(\Omega)}\nonumber\\
 &=\left(\int_0^{t_0-\varepsilon+s}\abs{\frac{u^{k-\beta-1}}{(t_0-\varepsilon+s+u)^{m+k}}}du+
     \int_{t_0-\varepsilon+s}^\infty\abs{\frac{u^{k-\beta-1}}{(t_0-\varepsilon+s+u)^{m+k}}}du\right)
      \left\|f\right\|_{L_{\mathbb{B}}^p(\Omega)}\\
 &\leq C\frac{1}{(t_0-\varepsilon+s)^{\beta+m}}\left\|f\right\|_{L_{\mathbb{B}}^p(\Omega)}
 <\infty.\nonumber
 \end{align}
Combining \eqref{02.11} and \eqref{02.10}, we know that
$\displaystyle
\partial_t^m\left({\partial_{t+s}^k}\mathcal{P}_{t+s+u}(f)u^{k-\beta-1}\right)$
is controlled by an integrable function. Hence we can interchange
the order of the inner integration and the partial derivative
$\partial_t^m$ in \eqref{02.9} to obtain
\begin{align}\label{02.12}
&\frac{e^{-i\pi(m-\alpha)}}{\mathbf{\Gamma}(m-\alpha)}\int_0^\infty{\partial_t^m}
    \left(\partial_{t+s}^\beta\mathcal{P}_{t+s}f\right)s^{m-\alpha-1}ds\nonumber\\
&=\frac{e^{-i\pi(m+k-\alpha-\beta)}}{\mathbf{\Gamma}(m-\alpha)\mathbf{\Gamma}(k-\beta)}\int_0^\infty
    \int_0^\infty{\partial_t^m}\partial_{t+s}^k
    \mathcal{P}_{t+s+u}(f)u^{k-\beta-1}dus^{m-\alpha-1}ds\nonumber\\
&=\frac{e^{-i\pi(m+k-\alpha-\beta)}}{\mathbf{\Gamma}(m-\alpha)\mathbf{\Gamma}(k-\beta)}\int_0^\infty
    \int_0^\infty{\partial_t^{m+k}}\mathcal{P}_{t+s+u}(f)u^{k-\beta-1}dus^{m-\alpha-1}ds\nonumber\\
&=\frac{e^{-i\pi(m+k-\alpha-\beta)}}{\mathbf{\Gamma}(m-\alpha)\mathbf{\Gamma}(k-\beta)}\int_0^\infty
    \int_s^\infty{\partial_t^{m+k}}\mathcal{P}_{t+w}(f)(w-s)^{k-\beta-1}dws^{m-\alpha-1}ds\\
&=\frac{e^{-i\pi(m+k-\alpha-\beta)}}{\mathbf{\Gamma}(m-\alpha)\mathbf{\Gamma}(k-\beta)}\int_0^\infty
    \int_0^w{\partial_t^{m+k}}\mathcal{P}_{t+w}(f)(w-s)^{k-\beta-1}s^{m-\alpha-1}dsdw\nonumber\\
&=\frac{e^{-i\pi(m+k-\alpha-\beta)}\mathbf{B}(m-\alpha,k-\beta)}{\mathbf{\Gamma}(m-\alpha)\mathbf{\Gamma}(k-\beta)}
    \int_0^\infty {\partial_t^{m+k}}\mathcal{P}_{t+w}(f)w^{k+m-\alpha-\beta-1}dw\nonumber\\
&=\frac{e^{-i\pi(m+k-\alpha-\beta)}}{\mathbf{\Gamma}(m+k-\alpha-\beta)}
    \int_0^\infty {\partial_t^{m+k}}\mathcal{P}_{t+w}(f)w^{k+m-\alpha-\beta-1}dw\nonumber.
\end{align}
Since $m-1\leq \alpha<m$ and $k-1\leq \beta <k,$ $m+k-2\leq
\alpha+\beta <m+k.$ If $m+k-1\leq \alpha+\beta <m+k,$ we have
\begin{equation}\label{02.14}
\frac{e^{-i\pi(m+k-\alpha-\beta)}}{\mathbf{\Gamma}(m+k-\alpha-\beta)}\int_0^\infty
{\partial_t^{m+k}}\mathcal{P}_{t+w}(f)w^{k+m-\alpha-\beta-1}dw
={\partial_t^{\alpha+\beta}}\mathcal{P}_tf.
\end{equation}
If $m+k-2\leq \alpha+\beta <m+k-1,$ then integrating by parts, we
get
\begin{multline}\label{02.13}
\frac{e^{-i\pi(m+k-\alpha-\beta)}}{\mathbf{\Gamma}(m+k-\alpha-\beta)}
    \int_0^\infty {\partial_t^{m+k}}\mathcal{P}_{t+w}(f)w^{k+m-\alpha-\beta-1}dw\\
=\frac{e^{-i\pi(m+k-1-\alpha-\beta)}}{\mathbf{\Gamma}(m+k-1-\alpha-\beta)}
    \int_0^\infty
    {\partial_t^{m+k-1}}\mathcal{P}_{t+w}(f)w^{k+m-\alpha-\beta-2}dw= {\partial_t^{\alpha+\beta}}\mathcal{P}_tf.
\end{multline}
So, combining \eqref{02.9} and \eqref{02.12}--\eqref{02.14}, we get
\begin{equation*}\label{02.15}
\frac{e^{-i\pi(m-\alpha)}}{\mathbf{\Gamma}(m-\alpha)}\int_0^\infty{\partial_t^m}
    \left(\partial_{t+s}^\beta\mathcal{P}_{t+s}f\right)s^{m-\beta-1}ds
    =\partial_t^{\alpha+\beta}\mathcal{P}_{t}f,
\end{equation*}
for any $\displaystyle f\in \bigcup\limits_{1\leq p\leq
\infty}L_{\mathbb{B}}^p \left(\Omega\right).$
\end{proof}

Write the spectral decomposition of the semigroup $\{\mathcal
{P}_t\}_{t\geq 0}$: for any $f\in L^2(\Omega)$
$$\mathcal {P}_tf=\int_0^\infty {e}^{-\lambda t}dE_f({\lambda}),$$
where ${E(\lambda})$ is a resolution of the identity. Thus
\begin{equation}\label{2.23}
{\partial_t^k}\mathcal{P}_tf=e^{-i\pi k}\int_{0^+}^\infty \lambda^k
e^{-\lambda t}dE_f({\lambda}), \quad k=1,2,\ldots.
 \end{equation}
 We have the
following proposition.
\begin{prop}
\label{prop1} Let $f\in L^2(\Omega)$ and $0<\alpha<\infty.$ We have
\begin{equation}\label{2.25}
\partial_t^\alpha \mathcal{P}_tf=e^{-i\pi\alpha} \int_{0^+}^\infty \lambda^\alpha
        e^{-\lambda t}  dE_f({\lambda}).
\end{equation}
\end{prop}
\begin{proof}
 By \eqref{2.1}
and \eqref{2.23}, we have
\begin{equation}\label{2.24}
\partial_t^{\alpha}\mathcal{P}_tf
  =\frac{e^{-i\pi\alpha}}{\mathbf{\Gamma}(k-\alpha)}\int_0^\infty
      \int_{0^+}^\infty\lambda^k
      e^{-(t+s)\lambda}dE_f({\lambda})s^{k-\alpha-1}ds,
\end{equation}
where $k$ is the smallest integer which is bigger than $\alpha.$
Then $\displaystyle \int_0^\infty \int_{0^+}^\infty\lambda^k
~e^{-(t+s)\lambda}\left|dE({\lambda})\right| s^{k-\alpha-1}ds$ is
absolutely convergent. And by Theorem \ref{thm2.1}, we know that the
integral in \eqref{2.1} is absolutely convergent in $L^2(\Omega).$
So by \eqref{2.24}, we get
\begin{align*}
\left\langle
\partial_t^{\alpha}\mathcal{P}_tf,\ g\right\rangle
&=  \left\langle
\frac{e^{-i\pi\alpha}}{\mathbf{\Gamma}(k-\alpha)}\int_0^\infty
     \int_{0^+}^\infty\lambda^k e^{-(t+s)\lambda}dE_f({\lambda})s^{k-\alpha-1}ds,\
     g\right\rangle \nonumber\\
&= \frac{e^{-i\pi\alpha}}{\mathbf{\Gamma}(k-\alpha)}
\int_0^\infty\left\langle
     \int_{0^+}^\infty\lambda^k e^{-(t+s)\lambda}dE_f({\lambda}),\
     g\right\rangle s^{k-\alpha-1}ds \nonumber\\
&= \frac{e^{-i\pi\alpha}}{\mathbf{\Gamma}(k-\alpha)} \int_0^\infty
     \int_{0^+}^\infty\lambda^k ~e^{-(t+s)\lambda}dE_{\langle f, g\rangle}(\lambda) s^{k-\alpha-1}ds \nonumber\\
&= \frac{e^{-i\pi\alpha}}{\mathbf{\Gamma}(k-\alpha)}
\int_{0^+}^\infty
     \int_0^\infty\lambda^k ~e^{-(t+s)\lambda}s^{k-\alpha-1}dsdE_{\langle f, g\rangle}(\lambda) \nonumber\\
&= \left\langle e^{-i\pi\alpha}\int_{0^+}^\infty  \lambda^\alpha
~e^{-t\lambda}dE_f({\lambda}),\  g\right\rangle, \quad \forall g\in
L^2(\Omega). \nonumber
\end{align*}
Hence we get \eqref{2.25}.
\end{proof}

\section{ Technical Results for Littlewood--Paley $g$-function}\label{tech g-function}

In this section, we will give some properties of the fractional
Littlewood--Paley $g$-function.
\begin{prop}\label{prop3}
Given a Banach space $\mathbb{B}$, $1< q<\infty,$ and
$0<\beta<\gamma,$ there exists a constant $C$ such that
\begin{equation}\label{02.16}
g_\beta^q(f)\leq Cg_\gamma^q(f), \quad \forall
f\in\bigcup\limits_{1\leq p\leq \infty}L_{\mathbb{B}}^p
\left(\Omega\right).
\end{equation}
\end{prop}
\begin{proof}
Assume that $f\in L_{\mathbb{B}}^p \left(\Omega\right)$ for some
$1\leq p\leq \infty.$ By Theorem \ref{thm2.2} and H\"{o}lder's
inequality, we have
\begin{align*}
&\big\|\partial_t^\beta \mathcal {P}_tf\big\|_\mathbb{B} \leq
\frac{1}{\mathbf{\Gamma}(\gamma-\beta)}\int_t^\infty
\big\|\partial_s^\gamma
         \mathcal {P}_sf\big\|_\mathbb{B}(s-t)^{\gamma-\beta-1}ds \\
   &\leq \frac{1}{\mathbf{\Gamma}(\gamma-\beta)}\left(\int_t^\infty \big\|\partial_s^\gamma
          \mathcal{P}_sf\big\|_\mathbb{B}^q(s-t)^{\gamma-\beta-1}s^{\gamma(q-1)}ds\right)^{\frac{1}{q}}
         \left(\int_t^\infty(s-t)^{\gamma-\beta-1}s^{-\gamma}ds\right)^{\frac{1}{q'}}.
\end{align*}
By changing variables, we have
\begin{align*}
\int_t^\infty(s-t)^{\gamma-\beta-1}s^{-\gamma}ds
   &= \int_t^\infty \left(1-\frac{t}{s}\right)^{\gamma-\beta-1}
        \left(\frac{t}{s}\right)^{\beta+1} t^{-\beta-1}ds   \\
   &= t^{-\beta}\int_0^1 (1-u)^{\gamma-\beta-1}u^{\beta-1}du =t^{-\beta}\mathbf{B}(\gamma-\beta, \beta).
\end{align*}
So we have
\begin{align}\label{2.11}
   \left\|\partial_t^\beta \mathcal
          {P}_tf\right\|_\mathbb{B}
   &\leq\frac{1}{\mathbf{\Gamma}(\gamma-\beta)}
          \left(t^{-\beta}\mathbf{B}(\gamma-\beta,\beta)\right)^{\frac{1}{q'}}\left(\int_t^\infty
          \big\|\partial_s^\gamma \mathcal{P}_sf\big\|_\mathbb{B}^q(s-t)^
          {\gamma-\beta-1}s^{\gamma(q-1)}ds\right)^{\frac{1}{q}}\\
   &= \frac{(\mathbf{B}(\gamma-\beta,\beta))^{\frac{1}{q'}}}{\mathbf{\Gamma}(\gamma-\beta)}
           \left(t^{-\beta}\right)^{\frac{1}{q'}}
          \left(\int_t^\infty \left\|\partial_s^\gamma \mathcal {P}_sf
           \right\|_\mathbb{B}^q(s-t)^{\gamma-\beta-1}s^{\gamma(q-1)}ds\right)^{\frac{1}{q}}.\nonumber
\end{align}
Using Fubini's theorem, by \eqref{2.11} we get
\begin{align*}
\int_0^\infty \left\|t^\beta\partial_t^\beta \mathcal
          {P}_tf\right\|_\mathbb{B}^q\frac{dt}{t}
   &\leq \frac{(\mathbf{B}(\gamma-\beta,\beta))^{\frac{q}{q'}}}{\mathbf{\Gamma}(\gamma-\beta)^q}
           \int_0^\infty t^{\beta q}\left(t^{-\beta}\right)^{\frac{q}{q'}}
          \int_t^\infty \norm{\partial_s^\gamma \mathcal {P}_sf}_\mathbb{B}^q
          (s-t)^{\gamma-\beta-1}s^{\gamma (q-1)}ds\frac{dt}{t}\nonumber\\
   &=    \frac{(\mathbf{B}(\gamma-\beta,\beta))^{q-1}}{\mathbf{\Gamma}(\gamma-\beta)^q}
          \int_0^\infty s^{\gamma (q-1)}\norm{\partial_s^\gamma \mathcal {P}_sf}_\mathbb{B}^q
         \int_0^s t^{\beta-1}(s-t)^{\gamma-\beta-1}dtds\nonumber\\
   &= \left(\frac{\mathbf{\Gamma}(\beta)}{\mathbf{\Gamma}(\gamma)}\right)^q
          \int_0^\infty \big\|s^\gamma\partial_s^\gamma \mathcal {P}_sf\big\|_\mathbb{B}^q\frac{ds}{s}.
\end{align*}
Hence we get the inequality \eqref{02.16} with the constant
$\displaystyle
C=\frac{\mathbf{\Gamma}(\beta)}{\mathbf{\Gamma}(\gamma)}$.
\end{proof}

In the following, we  shall need the theory of Calder\'{o}n--Zygmund
on $\mathbb{R}^n.$ So we should recall briefly the definition of the
Calder\'{o}n--Zygmund operator. Given two Banach spaces
$\mathbb{B}_1$ and $\mathbb{B}_2,$ let $T$ be a linear operator.
Then we call that $T$ is a Calder\'{o}n--Zygmund operator on
$\mathbb{R}^n,$ with associated Calder\'{o}n--Zygmund kernel $K$ if
$T$ maps $L_{c, \mathbb{B}_1}^{\infty},$ the space of the
essentially bounded $\mathbb{B}_1$-valued functions on
$\mathbb{R}^n$ with compact support, into the space of
$\mathbb{B}_2$-valued and strongly measurable functions on
$\mathbb{R}^n,$ and for any function $f\in L_{c,
\mathbb{B}_1}^{\infty}$ we have
$$Tf(x)=\int_{\mathbb{R}^n} K(x,y)f(y)dy, \quad \hbox{a.e.}\ x\in \mathbb{R}^n\
\hbox{outside the support of}\ f,$$ where the kernel $K(x, y)$ is a
regular kernel, that is, $K(x,y)\in \mathcal {L}\left(\mathbb{B}_1,
\mathbb{B}_2\right)$ satisfies $\displaystyle
\left\|K(x,y)\right\|\leq C\frac{1}{|x-y|^n}$ and
     $\displaystyle \left\|\bigtriangledown_x K(x, y)\right\|
             +\left\|\bigtriangledown_y K(x, y)\right\|\leq C\frac{1}{|x-y|^{n+1}}, \ \hbox{for any } x,y\in \mathbb{R}^n \hbox{ and } x\neq y,$
where as usual $\bigtriangledown_x=\left({\partial_{x_1}}, \cdots,
{\partial_{x_n}}\right)$.

Let us recall the $\mathbb{B}$-valued $BMO$ and $H^1$ spaces on
$\mathbb{R}^n.$ It is well known that
\begin{equation*}
BMO_{\mathbb{B}}(\mathbb{R}^n)=\set{f\in L_{\mathbb{B},
\textup{loc}}^1(\mathbb{R}^n): \sup_{\hbox{cubes } Q\subset
\mathbb{R}^n}\frac{1}{|Q|}\int_Q\left\|f(x)-\frac{1}{|Q|}\int_Qf(y)dy\right\|_{\mathbb{B}}dx<\infty}.
\end{equation*}
 The $\mathbb{B}$-valued $H^1$ space is defined in the
atomic sense. We say that a function $\displaystyle a\in
L_{\mathbb{B}}^\infty (\mathbb{R}^n)$ is a $\mathbb{B}$-valued atom
if there exists a cube $Q\subset\mathbb{R}^n$ containing the support
of $a,$ and  such that $\|a\|_{L_{\mathbb{B}}^\infty
(\mathbb{R}^n)}\leq |Q|^{-1}$ and $\displaystyle \int_Qa(x)dx=0.$
Then, we can define $H_{\mathbb{B}}^1\left(\mathbb{R}^n\right)$ as
\begin{equation*}
H_{\mathbb{B}}^1\left(\mathbb{R}^n\right)=\set{f: f=\sum_i\lambda_i
a_i,\ a_i \hbox{ are } \mathbb{B}\hbox{-valued atoms and }
\sum_i|\lambda_i|<\infty}.
\end{equation*}
 We define $\displaystyle
\|f\|_{H_\mathbb{B}^1\left(\mathbb{R}^n\right)}=\inf\Big\{\sum_i
|\lambda_i|\Big\},$ where the infimum runs over all those such
decompositions.

\begin{rem}{\cite[Theorem
4.1]{MTX}}\label{CZ} Given a pair of Banach spaces $\mathbb{B}_1$
and $\mathbb{B}_2,$ let $T$ be a Calder\'{o}n--Zygmund operator on
$\mathbb{R}^n$ with regular vector-valued kernel. Then the following
statements are equivalent:
\begin{enumerate}
\item[(i)] $T$ maps $L_{c, \mathbb{B}_1}^\infty(\mathbb{R}^n)$
into $BMO_{\mathbb{B}_2}(\mathbb{R}^n).$
\item[(ii)] $T$ maps $H_{\mathbb{B}_1}^1(\mathbb{R}^n)$ into
    $L^1_{\mathbb{B}_2}(\mathbb{R}^n).$
\item[(iii)] $T$ maps $L_{\mathbb{B}_1}^p(\mathbb{R}^n)$
into $L^p_{\mathbb{B}_2}(\mathbb{R}^n)$ for any (or, equivalently,
for some) $p\in (1, \infty).$
\item[(iv)] $T$ maps $BMO_{c,
\mathbb{B}_1}(\mathbb{R}^n)$ into
$BMO_{\mathbb{B}_2}(\mathbb{R}^n).$
\item[(v)] $T$ maps
$L_{\mathbb{B}_1}^1(\mathbb{R}^n)$ into $L^{1,
\infty}_{\mathbb{B}_2}(\mathbb{R}^n).$
\end{enumerate}
\end{rem}

\begin{prop}
\label{prop5} Given a Banach space $\mathbb{B}$, $1< q<\infty,$ and
$0<\alpha<\infty$,  $g_\alpha^q(f)$ can be expressed as an
$L^q_\mathbb{B}(\mathbb{R}_+,\frac{dt}{t})$-norm of a
Calder\'{o}n--Zygmund operator on $\mathbb{R}^n$ with regular
vector-valued kernel.
\end{prop}
\begin{proof}
Assume that $m-1\leq \alpha<m$ for some positive integer $m.$ For
any $f\in S(\mathbb{R}^n)\otimes\mathbb{B}$, we have
$$g_\alpha^q(f)(x)
=\|t^\alpha \partial_t^\alpha \mathcal
   {P}_tf(x)\|_{L_\mathbb{B}^q(\mathbb{R}_+,\frac{dt}{t})}
=\left\|\int_{\mathbb{R}^n}K_t(x-y)f(y)dy\right\|_{L_\mathbb{B}^q(\mathbb{R}_+,\frac{dt}{t})},$$
with
$$K_t(x-y)=\frac{\mathbf{\Gamma}(\frac{n+1}{2})}{\pi^{\frac{n+1}{2}}\mathbf{\Gamma}(m-\alpha)}
         t^\alpha \int_0^\infty{\partial_t^m}
         \left(\frac{t+s}{((t+s)^2+|x-y|^2)^{\frac{n+1}{2}}}\right){s^{m-\alpha-1}}{ds},\quad
          x, y\in \mathbb{R}^n,  x\neq y, t>0.$$
It can be proved that
\begin{equation*}
\|K_t(x-y)\|_{\mathcal{L}\big(\mathbb{B},
     \ L_\mathbb{B}^q(\mathbb{R}_+, \frac{dt}{t})\big)}
    \le  C\frac{1}{|x-y|^n},
\end{equation*}
and
\begin{equation*}
\label{2.8} \left\|\nabla_y K_t(x-y)\right\|_{\mathcal
{L}\left(\mathbb{B},L_\mathbb{B}^q\left(\mathbb{R}_+,\frac{dt}{t}\right)\right)}
+ \left\|\nabla_x K_t(x-y)\right\|_{\mathcal
{L}\left(\mathbb{B},L_\mathbb{B}^q\left(\mathbb{R}_+,\frac{dt}{t}\right)\right)}\leq
C\frac{1}{|x-y|^{n+1}},
\end{equation*}
for any $x, y\in \mathbb{R}^n, x\neq y.$ We leave the details of the
proof to the reader. A sketch of it can be found in \cite{JorFRTT}.
\end{proof}

\begin{prop}\label{prop4}
Let $\mathbb{B}$ be a Banach space  which is of Lusin cotype $q$,
$2\leq q<\infty$. Then for every symmetric diffusion semigroup
$\{\mathcal{T}_t\} _{t\geq 0}$ with subordinated semigroup
$\{\mathcal {P}_t\}_{t\geq 0}$ and for every (or, equivalently, for
some) $p\in (1,\infty)$, there is a constant $C$ such that
\begin{equation}\label{02.17}
\|g_{k}^q(f)\|_{L^p(\Omega)}\leq C\|f\|_{L_\mathbb{B}^p(\Omega)},
\quad k=1,2, \ldots, \quad \forall f\in L_\mathbb{B}^p(\Omega).
\end{equation}
Moreover, for any $0<\alpha<\infty,$ if
\begin{equation}\label{02.18}
 \|g_{\alpha}^q(f)\|_{L^p(\Omega)}\leq
C\|f\|_{L_\mathbb{B}^p(\Omega)}, \quad \forall f\in
L_\mathbb{B}^p(\Omega),
\end{equation}
then we have
 \begin{equation}\label{02.19}
\|g_{k\alpha}^q(f)\|_{L^p(\Omega)}\leq
C\|f\|_{L_\mathbb{B}^p(\Omega)}, \quad k=1,2, \ldots, \quad \forall
f\in L_\mathbb{B}^p(\Omega).
\end{equation}
\end{prop}
\begin{proof} For the case $k=1$, the inequality \eqref{02.17} have
been proved in \cite{MTX}. We only need prove the cases
$k=2,3,\ldots.$ We can prove it by induction. Assume that the
inequality \eqref{02.17} is true for some $1\leq k \in \mathbb{Z}.$
Let us prove that it is true for $k+1$ also. Since the inequality
\eqref{02.17} is true for $k$,  we know that the following operator
$$T: L_\mathbb{B}^q(\mathbb{R}^n)\longrightarrow
L_{L_\mathbb{B}^q\left(\mathbb{R}_+,
 \frac{dt}{t}\right)}^q(\mathbb{R}^n),$$
$$Tf(x,t)=t^k\partial_t^k\mathcal {P}_tf(x),\qquad \forall
f\in L_\mathbb{B}^q(\mathbb{R}^n)$$ is bounded. By Fubini's theorem
we know that the operator $$\tilde{T}:
L_{L_\mathbb{B}^q\left(\mathbb{R}_+ ,
\frac{dt}{t}\right)}^q(\mathbb{R}^n)\longrightarrow
    L^q_{L_\mathbb{B}^q\left( \frac{ds}{s}\frac{dt}{t} \right)}
    \left(\mathbb{R}^n\right),$$
$$\tilde{T}F(x,s,t)=s \partial_s \mathcal {P}_s(F)(x,t),\qquad \forall F(x,t)
    \in L_{L_\mathbb{B}^q\left(\mathbb{R}_+, \frac{dt}{t}\right)}^q(\mathbb{R}^n)$$
is also bounded. Since $\tilde{T}$ can be expressed as a
Calder\'{o}n--Zygmund operator with regular vector-valued  kernel,
by Remark \ref{CZ} we get that $\displaystyle \tilde{T}:
L_{L_\mathbb{B}^q\left(\mathbb{R}_+ ,
\frac{dt}{t}\right)}^p(\mathbb{R}^n)\longrightarrow
    L_{L_\mathbb{B}^q\left( \frac{ds}{s}\frac{dt}{t} \right)}^p(\mathbb{R}^n)$ is bounded for
any $1<p<\infty.$ Hence, by Theorem 5.2 of \cite{MTX}, we know that
$L_\mathbb{B}^q\left(\mathbb{R}_+ , \frac{ds}{s}\right)$ is of Lusin
cotype $q$.

Now given a symmetric diffusion semigroup $\{\mathcal{T}_t\} _{t\geq
0}$ with subordinated semigroup $\{\mathcal {P}_t\}_{t\geq 0}$. As
$\mathbb{B}$ is of Lusin cotype $q$ and $
L_\mathbb{B}^q\left(\mathbb{R}_+ , \frac{ds}{s}\right)$ also is of
Lusin cotype $q$, we get that $T$ is bounded from $\displaystyle
L_{\mathbb{B}}^p\left(\Omega\right)$ to $\displaystyle
L_{L_{\mathbb{B}}^q\left(\mathbb{R}_+,
\frac{dt}{t}\right)}^p\left(\Omega \right)$ and $\tilde{T}$ is
bounded from $L_{L_{\mathbb{B}}^q\left(\mathbb{R}_+,
\frac{dt}{t}\right)}^p\left(\Omega \right)$  to $\displaystyle
L_{L_\mathbb{B}^q\left( \frac{ds}{s}\frac{dt}{t} \right)}^p\left(\Omega\right),$
for any $1<p<\infty.$ So the operator $\tilde{T}\circ T$ is bounded
from $L_{\mathbb{B}}^p\left(\Omega\right)$ to
$L_{L_\mathbb{B}^q\left( \frac{ds}{s}\frac{dt}{t} \right)}^p\left(\Omega\right),$
for any $1<p<\infty,$ and by \eqref{02.6} we have
\begin{align}\label{formula}
\tilde{T}\circ Tf(x,t,s)&= \tilde{T}(T f(x,t))(s)
   = \tilde{T}(t^k\partial_t^k \mathcal {P}_tf(x))(s)\nonumber\\
  &= s \partial_s \mathcal{P}_s(t^k \partial_t^k \mathcal{P}_tf)(x)
  = s t^k \partial_s \partial_t^k \mathcal{P}_s\mathcal{P}_tf(x)\\
  &= s t^k \partial_s \partial_t^k \mathcal{P}_{s+t}f(x)
  = s t^k \partial_u^{k+1} \mathcal{P}_uf\big|_{u=t+s}(x).\nonumber
\end{align}
So there exists a constant $C$ such that
\begin{align}\label{03.3}
\left\|f\right\|_{L_{\mathbb{B}}^p\left(\Omega\right)}^p
  &\geq C\left\|\tilde{T}\circ Tf\right\|_{L_{L_\mathbb{B}^q\left( \frac{ds}{s}\frac{dt}{t} \right)}^p\left(\Omega\right)}^p\nonumber\\
  &= C\left\|s t^k \partial_u^{k+1} \mathcal{P}_uf|_{u=t+s}(x)\right\|_{L_{{L_\mathbb{B}^q\left( \frac{ds}{s}\frac{dt}{t} \right)}}^p\left(\Omega\right)}^p\nonumber\\
  &= C\left\|\left(\int_0^\infty \int_0^\infty\left\|st^k
      \partial_u^{k+1}\mathcal {P}_uf\big|_{u=t+s}(x)
      \right\|_\mathbb{B}^q \frac{ds}{s}\frac{dt}{t}\right)^{\frac{1}{q}}
      \right\|_{L^p\left(\Omega\right)}^p\nonumber\\
 &= C\left\|\left(\int_0^\infty \int_t^\infty t^{kq}(s-t)^{ q}\big\|\partial_s^{k+1}
      \mathcal{P}_sf\big\|_\mathbb{B}^q\frac{ds}{s-t}\frac{dt}{t}\right)^{\frac{1}{q}}
      \right\|_{L^p\left(\Omega\right)}^p  \\
 &=  C\left\|\left(\int_0^\infty \big\|\partial_s^{k+1}\mathcal {P}_sf\big\|_\mathbb{B}^q
     \int_0^st^{kq-1}(s-t)^{q-1}dtds\right)^{\frac{1}{q}}\right\|_{L^p\left(\Omega\right)}^p\nonumber\\
 &= C(\mathbf{B}(kq, q))^{\frac{p}{q}}\left\|\left(\int_0^\infty s^{(k+1)q}\left\|\partial_s^{k+1}
       \mathcal{P}_sf\right\|_\mathbb{B}^q \frac{ ds}{s} \right)^{\frac{1}{q}}\right\|_{L^p\left(\Omega\right)}^p\nonumber\\
 &= C(\mathbf{B}(kq, q))^{\frac{p}{q}}\left\|g_{k+1}^q(f)\right\|_{L^p\left(\Omega\right)}^p.\nonumber
\end{align}
Whence $$\|g_{k+1}^q(f)\|_{L^p(\Omega)}\leq
C\|f\|_{L_\mathbb{B}^p(\Omega)}, \quad \forall f\in
L_\mathbb{B}^p(\Omega).$$ Then we get the inequality \eqref{02.17}
for any $k\in \mathbb{Z}_+$.

We can prove inequality \eqref{02.19} under the assumption
\eqref{02.18} with the similar argument as above. The only
difference is that we should define $T$ by
$$Tf(x,t)=t^{k\alpha}\partial_t^{k\alpha}\mathcal{P}_tf(x), \qquad \forall
f\in L_\mathbb{B}^q(\mathbb{R}^n),$$ and define $\tilde{T}$ by
$$\tilde{T}F(x,s,t)=s^{\alpha}\partial_s^{\alpha}\mathcal{P}_sF(x,t),\qquad \forall F(x,t)
    \in L_{L_\mathbb{B}^q\left(\mathbb{R}_+,
    \frac{dt}{t}\right)}^q(\mathbb{R}^n).$$ And by Proposition
    \ref{prop5} we know that in this case $\tilde{T}$
 can be expressed as a Calder\'{o}n--Zygmund operator also.
\end{proof}

The following theorem is proved in \cite{MTX} which we will use
later.
\begin{thm}{\cite[Theorem 3.2]{MTX}}\label{th3.2} Let $\mathbb{B}$ be a Banach space and $1<p,q< \infty.$ Let $h(x,t)$ be a function
in $L^p_{L^q_\mathbb{B}\left(\mathbb{R}_+,
\frac{dt}{t}\right)}(\Omega).$ Consider the operator $Q$ defined by
$\displaystyle Qh(x) = \int_0^\infty
\partial_t \mathcal{P}_t h(x,t)dt, \, x \in \Omega$. Then for nice
function $h$ we have
$$ \norm{g_1^q(Qh)}_{L^p(\Omega) } \le C_{p,q} \|h \|_{L^p_{L^q_\mathbb{B}\left(\mathbb{R}_+, \frac{dt}{t}\right)}\left(\Omega\right)},$$
where the constant $C_{p,q}$  depends only on $p$ and $q$.
\end{thm}

\section{Proofs of Theorem A and Theorem B} \label{g-function}

 Now we are in a position to prove Theorem A and Theorem B.
\begin{proof}[\bf Proof of Theorem A]
 $\textup{(i)}\Rightarrow \textup{(ii)}.$ Since $\mathbb{B}$ is of Lusin cotype
 $q$, by Proposition \ref{prop4} we have $$\left\|g_{k}^q(f)\right\|_{L^p(\Omega)}\leq
C\|f\|_{L_\mathbb{B}^p(\Omega)}, \quad k=1,2,\ldots, \quad \forall
f\in L_\mathbb{B}^p(\Omega).$$ Then, for any $\alpha>0$, there
exists $k\in \mathbb{N}$ such that $\alpha<k.$ By Proposition
\ref{prop3}, we have
$$\left\|g_{\alpha}^q(f)\right\|_{L^p(\Omega)}\leq C\left\|g_{k}^q(f)\right\|_{L^p(\Omega)}\leq
C\|f\|_{L_\mathbb{B}^p(\Omega)},\quad \forall f\in
L_\mathbb{B}^p(\Omega).$$

$\textup{(ii)}\Rightarrow \textup{(i)}.$ Since
$\left\|g_{\alpha}^q(f)\right\|_{L^p(\Omega)}\leq
C\|f\|_{L_\mathbb{B}^p(\Omega)}$ for any $f\in
L_\mathbb{B}^p(\Omega),$ by Proposition \ref{prop4} there exists an
integer $k$ such that $k\alpha>1$ and
$$\left\|g_{k\alpha}^q(f)\right\|_{L^p(\Omega)}\leq
C\|f\|_{L_\mathbb{B}^p(\Omega)}$$ for any $f\in
L_\mathbb{B}^p(\Omega).$ By Proposition \ref{prop3}, we have
$$\left\|g_1^q(f)\right\|_{L^p(\Omega)}\leq C\left\|g_{k\alpha}^q(f)\right\|_{L^p(\Omega)}\leq
C\|f\|_{L_\mathbb{B}^p(\Omega)}$$ for any $f\in
L_\mathbb{B}^p(\Omega).$ Hence, by Theorem 1.1, $\mathbb{B}$ is of
Lusin cotype $q.$
\end{proof}

\begin{proof}[\bf Proof of Theorem B] $\textup{(i)}\Rightarrow \textup{(ii)}.$
It is easy to deduce from \eqref{2.25} that for any $f, g \in
L^2(\Omega)$
\begin{equation}\label{02.22}
\int_{\Omega}(f-E_0(f))(g-E_0(g))d\mu=\frac{4^\alpha}
  {\mathbf{\Gamma}(2\alpha)}\int_{\Omega}\int_0^\infty
   (t^\alpha \partial_t^\alpha \mathcal{P}_tf)(t^\alpha \partial_t^\alpha \mathcal{P}_tg)\frac{dt}{t}d\mu.
\end{equation}
Now we use duality. Fix two functions $f\in L_\mathbb{B}^p(\Omega)$
and $g\in L_\mathbb{B^*}^{p'}(\Omega),$ where
$\frac{1}{p}+\frac{1}{p'}=1.$ Without loss of generality, we may
assume that $f$ and $g$ are in the algebraic tensor products
$\big(L^p(\Omega)\cap L^2(\Omega)\big)\otimes \mathbb{B}$ and
$\big(L^{p'}(\Omega)\cap L^2(\Omega)\big)\otimes \mathbb{B}^*,$
respectively. With $\langle\ \ ,\ \rangle$ denoting the duality
between $\mathbb{B}$ and $\mathbb{B}^*,$ we have
\begin{equation}
\label{2.17}\int_{\Omega}\left\langle f, g\right\rangle
d\mu=\int_{\Omega}\left\langle E_0(f), E_0(g)\right\rangle d\mu
  +\int_{\Omega}\left\langle f-E_0(f), g-E_0(g)\right\rangle d\mu.
\end{equation}
The first term on the right is easy to be estimated:
\begin{equation}
\label{2.18}\left|\int_{\Omega}\langle E_0(f), E_0(g)\rangle
d\mu\right|
  \leq \|E_0(f)\|_{L_\mathbb{B}^p(\Omega)}\|E_0(g)\|_{L_{\mathbb{B}^*}^{p'}(\Omega)}
  \leq \|E_0(f)\|_{L_\mathbb{B}^p(\Omega)}\|g\|_{L_{\mathbb{B}^*}^{p'}(\Omega)}.
\end{equation}
For the second one, by \eqref{02.22} and H\"{o}lder's inequality
\begin{align}\label{2.19}
\left|\int_{\Omega}\langle f-E_0(f), g-E_0(g)\rangle d\mu\right|
  &= \frac{4^\alpha}{\mathbf{\Gamma}(2\alpha)}
       \left|\int_{\Omega}\int_0^\infty\left\langle t^\alpha \partial_t^\alpha
       \mathcal{P}_tf,t^\alpha \partial_t^\alpha \mathcal{P}_tg \right\rangle \frac{dt}{t}d\mu\right|\nonumber\\
  &\le \frac{4^\alpha}{\mathbf{\Gamma}(2\alpha)}\int_{\Omega}\int_0^\infty\left\|t^\alpha \partial_t^\alpha
       \mathcal{P}_tf\right\|_\mathbb{B} \left\|t^\alpha \partial_t^\alpha
       \mathcal{P}_tg\right\|_{\mathbb{B}^*} \frac{dt}{t}d\mu \\
  &\le \frac{4^\alpha}{\mathbf{\Gamma}(2\alpha)}\big\|g_{\alpha}^q(f)\big\|_{L^p(\Omega)}
       \big\|g_{\alpha}^{q'}(g)\big\|_{L^{p'}(\Omega)}.\nonumber
\end{align}
Now since $\mathbb{B}$ is of Lusin type $q$, $\mathbb{B}^*$ is of
Lusin cotype $q'$. Thus by Theorem A,
\begin{equation}
\label{2.20} \big\|g_{\alpha}^{q'}(g)\big\|_{L^{p'}(\Omega)}\leq
C\left\|g\right\|_{L_{\mathbb{B}^*}^{p'}(\Omega)}.
\end{equation}
Combining \eqref{2.17}--\eqref{2.20}, we get
$$\Big|\int_{\Omega}\langle f, g\rangle d\mu\Big|\leq \Big( \left\|E_0(f)\right\|_{L_\mathbb{B}^p(\Omega)}
   +C \left\|g_{\alpha}^q(f)\right\|_{L^p(\Omega)} \Big)\left\|g\right\|_{L_{\mathbb{B}^*}^{p'}(\Omega)},$$
which gives (ii) by taking the supremum over all $g$ as above such
that $\|g\|_{L_{\mathbb{B}^*}^{p'}(\Omega)}\leq 1.$

$\textup{(ii)}\Rightarrow \textup{(i)}.$ We only need consider the
particular case on $\mathbb{R}^n$.  In this case, $E_0(f)=0$ for any
$f\in L_{\mathbb{B}}^p(\mathbb{R}^n).$ Assuming $p=q$ and $k-1\le
\alpha<k$ for some $k\in \mathbb{Z}_+$, by Proposition \ref{prop3}
we have
\begin{equation}\label{03.2}
\|f\|_{L_\mathbb{B}^q(\mathbb{R}^n)}\leq
C\left\|g_{\alpha}^q(f)\right\|_{L^q(\mathbb{R}^n)}\leq C
\left\|g_k^q(f)\right\|_{L^q(\mathbb{R}^n)}, \,
\end{equation} for
any $f\in L_{\mathbb{B}}^q(\mathbb{R}^n).$ By using (\ref{03.3}) and
(\ref{formula}), we have
\begin{equation*}
\left(\int_0^\infty
    s^{kq}\norm{\partial_s^k\mathcal{P}_sf}_{\mathbb{B}}^q\frac{ds}{s}\right)^{\frac{1}{q}}
 =C\left(\int_0^\infty \int_0^\infty s_1^qs_2^{(k-1)q}\norm{\partial_{s_2}^{k-1}\mathcal{P}_{s_2}
   \left(\partial_{s_1}\mathcal{P}_{s_1}\right)f}_{\mathbb{B}}^q\frac{ds_2}{s_2}\frac{ds_1}{s_1}\right)^{\frac{1}{q}}.
\end{equation*}
By iterating the argument, we can get
\begin{equation*}
\left(\int_0^\infty
    s^{kq}\norm{\partial_s^k\mathcal{P}_sf}_{\mathbb{B}}^q\frac{ds}{s}\right)^{\frac{1}{q}}
 =C\left(\int_0^\infty\cdots \int_0^\infty s_1^q\cdots
 s_k^q\norm{\partial_{s_1}\mathcal{P}_{s_1}\cdots\partial_{s_k}\mathcal{P}_{s_k} f}_{\mathbb{B}}^q
 \frac{ds_1}{s_1}\cdots \frac{ds_k}{s_k}\right)^{\frac{1}{q}}.
\end{equation*}
Therefore we can choose a function $\displaystyle b(x,s_1, \dots,
s_k) \in
L^q_{L_\mathbb{B}^q\big(\frac{dt_1}{t_1}\cdots\frac{dt_k}{t_k}\big)}\left(\mathbb{R}^n\right)$
of unit norm such that
$$\big\|g^{q'}_k(f)\big\|_{L^{q'}(\mathbb{R}^n)} =C\int_{\mathbb{R}^n}
\int_0^\infty \cdots \int_0^\infty \left\langle s_1\cdots s_k~
\partial_{s_1}\mathcal{P}_{s_1}\cdots
\partial_{s_k} \mathcal{P}_{s_k} f(x),~b(x,s_1, \dots, s_k) \right\rangle \frac{ds_1}{s_1} \cdots\frac{ds_k}{s_k}dx.$$
We may assume that $f$ and $b$ are nice enough to legitimate the
calculations below. By Fubini's theorem, H\"{o}lder's inequality and
\eqref{03.2}, we have
\begin{align}\label{4.7}
&\big\|g^{q'}_k(f)\big\|_{L^{q'}(\mathbb{R}^n)}\nonumber\\
&=C\int_{\mathbb{R}^n} \int_0^\infty \cdots \int_0^\infty
\big\langle s_1\cdots s_k~ \partial_{s_1}\mathcal{P}_{s_1}\cdots
\partial_{s_k} \mathcal{P}_{s_k} f(x),~b(x,s_1, \dots, s_k) \big\rangle \frac{ds_1}{s_1} \cdots\frac{ds_k}{s_k}dx \nonumber \\
&=  C \int_{\mathbb{R}^n} \left\langle f(x), \int_0^\infty \cdots
      \int_0^\infty  s_1 \cdots s_k~\partial_{s_1}\mathcal{P}_{s_1}\cdots
      \partial_{s_k} \mathcal{P}_{s_k} b(x,s_1, \dots,s_k)  \frac{ds_1}{s_1} \cdots\frac{ds_k}{s_k} \right\rangle dx \nonumber \\
&\le  C \|f\|_{L^{q'}_{\mathbb{B}^*}(\mathbb{R}^n)} \norm{
        \int_0^\infty\cdots  \int_0^\infty s_1\cdots s_k\partial_{s_1}\mathcal{P}_{s_1}\cdots \partial_{s_k} \mathcal{P}_{s_k} b(x,s_1,
        \dots, s_k)\frac{ds_1}{s_1} \cdots\frac{ds_k}{s_k} }_{L^{q}_{\mathbb{B}}(\mathbb{R}^n)}  \\
&\le  C \|f\|_{L^{q'}_{\mathbb{B}^*}(\mathbb{R}^n)} \norm{
       g_k^q\left(\int_0^\infty  \cdots \int_0^\infty s_1\cdots s_k
       \partial_{s_1}\mathcal{P}_{s_1}\cdots \partial_{s_k}\mathcal{P}_{s_k}   b(x,s_1, \dots,
       s_k)  \frac{ds_1}{s_1} \cdots\frac{ds_k}{s_k}\right)}_{L^{q}(\mathbb{R}^n)}
       \nonumber\\
&=: C \|f\|_{L^{q'}_{\mathbb{B}^*}(\mathbb{R}^n)}
\norm{g_k^q\left(G_k(b)\right)}_{L^{q}(\mathbb{R}^n)}, \nonumber
\end{align}
where $$G_k(b)=\int_0^\infty  \cdots \int_0^\infty s_1 \cdots s_k
       \partial_{s_1}\mathcal{P}_{s_1}\cdots \partial_{s_k}\mathcal{P}_{s_k}   b(x,s_1, \dots,
       s_k)  \frac{ds_1}{s_1} \cdots\frac{ds_k}{s_k},\quad k\in \mathbb{Z}_+.$$
 Using \eqref{03.3}, Fubini's theorem and {Theorem  \ref{th3.2}} repeatedly, we have
\begin{align}\label{4.8}
&\norm{ g_k^q\left(G_k(b)\right)}^q_{L^{q}(\mathbb{R}^n)}\\
 &\le {C\int_{\mathbb{R}^n}\int_0^\infty\cdots\int_0^\infty\Big\|
         t_1\partial_{t_1}\mathcal{P}_{t_1}\cdots t_k\partial_{t_k}\mathcal{P}_{t_k}\left(G_k(b)\right)\Big\|_{\mathbb{B}}^q
         \frac{dt_1}{t_1}\cdots\frac{dt_k}{t_k}}dx\nonumber\\
  &= {C\int_0^\infty\cdots\int_0^\infty\int_{\mathbb{R}^n}\int_0^\infty\norm{t_k\partial_{t_k}\mathcal{P}_{t_k}
        \left[\int_0^\infty s_k \partial_{s_k}\mathcal{P}_{s_k}\left(
        t_1\partial_{t_1}\mathcal{P}_{t_1}\cdots
        t_{k-1}\partial_{t_{k-1}}\mathcal{P}_{t_{k-1}}G_{k-1}(b)\right)\frac{ds_k}{s_k}\right]}_{\mathbb{B}}^q}\nonumber\\
   &\quad {\frac{dt_k}{t_k}dx\frac{dt_1}{t_1}\cdots\frac{dt_{k-1}}{t_{k-1}}}\nonumber\\
   &\le C\int_0^\infty\cdots\int_0^\infty\int_{\mathbb{R}^n}\int_0^\infty\norm{
          t_1\cdots t_{k-1}\partial_{t_1}\mathcal{P}_{t_1}\cdots\partial_{t_{k-1}}\mathcal{P}_{t_{k-1}}G_{k-1}(b)}_{\mathbb{B}}^q\frac{ds_k}{s_k}
          dx\frac{dt_1}{t_1}\cdots\frac{dt_{k-1}}{t_{k-1}}\nonumber\\
   &\vdots \nonumber\\
   &\le C\int_{\mathbb{R}^n}\int_0^\infty\cdots\int_0^\infty\norm{b(x,s_1,
          \dots,s_k)}_{\mathbb{B}}^q
          \frac{ds_1}{s_1}\cdots\frac{ds_k}{s_k}dx=C.\nonumber
\end{align}
Combining \eqref{4.7} and \eqref{4.8}, we get
$$\left\|g_k^{q'}(f)\right\|_{L^{q'}\left(\mathbb{R}^n\right)}\leq
C\|f\|_{L_{\mathbb{B}^*}^{q'}(\mathbb{R}^n)}.$$ By Theorem A,
$\mathbb{B}^*$ is of Lusin cotype $q'.$ Hence $\mathbb{B}$ is of
Lusin type $q.$

If $p \neq q$, it suffices to prove that the operator   $b
\rightarrow g_k^q(G_k(b))$ maps
$L^p_{L_\mathbb{B}^q\big(\frac{dt_1}{t_1}\cdots\frac{dt_k}{t_k}\big)}\left(\mathbb{R}^n\right)
$ into $L^p(\mathbb{R}^n).$ To that end  we shall use the theory of
vector-valued Calder\'on--Zygmund operators. We borrow this idea
from \cite{OX}. Let us consider the operator
$$T(b)(x,t_1,\ldots,t_k)=t_1 \partial_{t_1}\mathcal{P}_{t_1} \cdots t_k \partial_{t_k}\mathcal{P}_{t_k}\int_0^\infty \cdots \int_0^\infty
s_1\partial_{s_1}\mathcal{P}_{s_1} \cdots s_k
\partial_{s_k}\mathcal{P}_{s_k}b(x,s_1,\ldots,s_k)\frac{ds_1}{s_1}\cdots
\frac{ds_k}{s_k}.$$ Clearly,
$$\norm{T(b)(x,t_1,\ldots,t_k)}_{L^p_{L_\mathbb{B}^q\big(\frac{dt_1}{t_1}\cdots
\frac{dt_k}{t_k}\big)}\left(\mathbb{R}^n\right)}=\norm{g_k^q\left(G_k(b)\right)}_{L^{p}(\mathbb{R}^n)}.$$
 Therefore it is enough to prove $$T:
L^p_{L_\mathbb{B}^q\big(\frac{dt_1}{t_1}\cdots\frac{dt_k}{t_k}\big)}\left(\mathbb{R}^n\right)\longrightarrow
L^p_{L_\mathbb{B}^q\big(\frac{ds_1}{s_1}\cdots\frac{ds_k}{s_k}\big)}\left(\mathbb{R}^n\right).$$
Hence as we already know that $T$ is bounded in the case $p=q$, in order to get the case $p\neq q$
it suffices to show  that the kernel of  $T$ satisfies
the standard estimates, see Remark \ref{CZ}. For simply and essentially, we only need
consider the case when $k=2.$ So \begin{align*} T(b)(x,t_1,t_2)&=t_1
t_2
\partial_{t_1}\mathcal{P}_{t_1}
\partial_{t_2}\mathcal{P}_{t_2}\int_0^\infty \int_0^\infty
s_1  s_2\partial_{s_1}\mathcal{P}_{s_1}
\partial_{s_2}\mathcal{P}_{s_2}(b)(x,s_1, s_2)\frac{ds_1}{s_1}
\frac{ds_2}{s_2}\\
&= \int_{\mathbb{R}^n}\int_0^\infty \int_0^\infty t_1 t_2
\partial_{t_1}\mathcal{P}_{t_1}
\partial_{t_2}\mathcal{P}_{t_2}s_1
s_2\partial_{s_1}\mathcal{P}_{s_1}
\partial_{s_2}\mathcal{P}_{s_2}(x-y)b(y,s_1, s_2)\frac{ds_1}{s_1}
\frac{ds_2}{s_2}dy.
\end{align*} Then the operator-valued kernel $K(x)$
is $\displaystyle \int_0^\infty \int_0^\infty t_1 t_2
\partial_{t_1}\mathcal{P}_{t_1}
\partial_{t_2}\mathcal{P}_{t_2}s_1
s_2\partial_{s_1}\mathcal{P}_{s_1}
\partial_{s_2}\mathcal{P}_{s_2}(x)\frac{ds_1}{s_1}
\frac{ds_2}{s_2}.$ For any $b(s_1, s_2)\in
 L^q_{\mathbb{B}}\left(\frac{ds_1}{s_1}\frac{ds_2}{s_2}\right)$
 with unit norm, we
have
\begin{align*}
\norm{K(x)b}_{\mathbb{B}}&=\norm{\int_0^\infty \int_0^\infty
t_1t_2\partial_{t_1}\mathcal{P}_{t_1}\partial_{t_2}\mathcal{P}_{t_2}s_1
s_2\partial_{s_1}\mathcal{P}_{s_1}
\partial_{s_2}\mathcal{P}_{s_2}(x)b(s_1,s_2)\frac{ds_1}{s_1}
\frac{ds_2}{s_2}}_{\mathbb{B}}\\
&= \norm{\int_0^\infty \int_0^\infty t_1 t_2s_1
s_2\partial^4_{u}\mathcal{P}_{u}(x)\Big|_{u=t_1+t_2+s_1+s_2}b(s_1,s_2)\frac{ds_1}{s_1}
\frac{ds_2}{s_2}}_{\mathbb{B}}\\
&\le \int_0^\infty \int_0^\infty t_1 t_2s_1
s_2\partial^4_{u}\mathcal{P}_{u}(x)\Big|_{u=t_1+t_2+s_1+s_2}\norm{b(s_1,s_2)}_{\mathbb{B}}\frac{ds_1}{s_1}
\frac{ds_2}{s_2}\\
&\le
\norm{b}_{L^q_{\mathbb{B}}\left(\frac{ds_1}{s_1}\frac{ds_2}{s_2}\right)}\left\{\int_0^\infty
\int_0^\infty \left(t_1 t_2s_1
s_2\partial^4_{u}\mathcal{P}_{u}(x)\Big|_{u=t_1+t_2+s_1+s_2}\right)^{q'}\frac{ds_1}{s_1}
\frac{ds_2}{s_2}\right\}^{\frac{1}{q'}}\\
&\le C \left\{\int_0^\infty \int_0^\infty \left(\frac{t_1 t_2s_1
s_2}{(t_1+t_2+s_1+s_2+|x|)^{n+4}}\right)^{q'}\frac{ds_1}{s_1}
\frac{ds_2}{s_2}\right\}^{\frac{1}{q'}}\\
&\le C \frac{t_1t_2}{(t_1+t_2+|x|)^{n+2}}.
\end{align*}
Therefore,
\begin{multline*}
\norm{K(x)b}_{L^q_{\mathbb{B}}\left(\frac{dt_1}{t_1}\frac{dt_2}{t_2}\right)}=\left(\int_0^\infty\int_0^\infty
    \norm{K(x)b}^q_{\mathbb{B}}\frac{dt_1}{t_1}\frac{dt_2}{t_2}\right)^{\frac{1}{q}}\\
\le C\left(\int_0^\infty\int_0^\infty
    \left(\frac{t_1t_2}{(t_1+t_2+|x|)^{n+2}}\right)^q
     \frac{dt_1}{t_1}\frac{dt_2}{t_2}\right)^{\frac{1}{q}}
\le \frac{C}{|x|^n}.
\end{multline*}
Similarly, we can show that $$\norm{\nabla K(x)}\le
\frac{C}{|x|^{n+1}}.$$ Therefore, $K$ is a regular vector-valued
kernel and the proof is finished.
\end{proof}

\section{Poisson Semigroup on $\mathbb{R}^n$}\label{onRn}

In this section, we devote to study the fractional area function and
the fractional $g_{\lambda}^{*}$-function on $\mathbb{R}^n$ in the
vector-valued case. Our main goal is to prove the analogous results
with Theorem A and Theorem B related to these two functions on
$\mathbb{R}^n.$

Let $\mathbb{B}$ be a Banach space, $0<\alpha<\infty,$ $\lambda>1,$
and $1<q<\infty.$ We define the fractional area function on
$\mathbb{R}^n$  as
$$S_\alpha^q(f)(x)=\left(\iint_{\Gamma(x)}\left\|t^\alpha\partial_t^\alpha
\mathcal{P}_tf(y)\right\|_\mathbb{B}^q\frac{dydt}{t^{n+1}}\right)^\frac{1}{q},
\quad \forall f\in \bigcup\limits_{1\leq p\leq
\infty}L_\mathbb{B}^p\left(\mathbb{R}^n\right),$$ where
$\displaystyle\Gamma(x)=\left\{(y, t)\in \mathbb{R}_+^{n+1}: |x-y|<t
\right\},$ and define the fractional $g_\lambda^*$-function on
$\mathbb{R}^n$  as
$$g_{\lambda, \alpha}^{q,*}(f)(x)=\left(\iint_{\mathbb{R}_+^{n+1}}
\left(\frac{t}{|x-y|+t}\right)^{\lambda n}
\left\|t^\alpha\partial_t^\alpha \mathcal
{P}_tf(y)\right\|_\mathbb{B}^q\frac{dydt}{t^{n+1}}\right)^{\frac{1}{q}},
\quad \forall f\in \bigcup\limits_{1\leq p\leq
\infty}L_\mathbb{B}^p\left(\mathbb{R}^n\right).$$

The following proposition demonstrate that the vector-valued
fractional area function $S_\alpha^q$ can be treated as an $
L^q_\mathbb{B}(\Gamma(0),\frac{dydt}{t^{n+1}})$-norm of a
Calder\'{o}n--Zygmund operator.

\begin{prop}
\label{lem3}  Given a Banach space $\mathbb{B}$, $1< q<\infty$ and
$0<\alpha<\infty$, then $S_\alpha^q(f)$ can be expressed as  an
$L^q_\mathbb{B}(\Gamma(0),\frac{dydt}{t^{n+1}})$-norm of  a
Calder\'{o}n--Zygmund operator on $\mathbb{R}^n$ with regular
vector-valued kernel.
\end{prop}
\begin{proof}
Assume that $m-1\leq \alpha<m$ for some positive integer $m.$ For
any $f\in S(\mathbb{R}^n)\otimes\mathbb{B}$, by changing of
variables we have
$$S_\alpha^q(f)(x)
=\left\|t^\alpha \partial_t^\alpha \mathcal
{P}_tf(x+y)\right\|_{L_\mathbb{B}^q\left(\Gamma(0),
\frac{dydt}{t^{n+1}}\right)}
=\left\|\int_{\mathbb{R}^n}K_{y,t}(x,z)f(z)dz\right\|_{L_\mathbb{B}^q\left(\Gamma(0),
\frac{dydt}{t^{n+1}}\right)},$$ where
$$K_{y,t}(x,z)=\frac{\mathbf{\Gamma}(\frac{n+1}{2})}{\pi^{\frac{n+1}{2}}\mathbf{\Gamma}(m-\alpha)}
         t^\alpha \int_0^\infty{\partial_t^m}
         \left(\frac{t+s}{\left((t+s)^2+|x+y-z|^2\right)^{\frac{n+1}{2}}}\right)
         {s^{m-\alpha-1}}{ds},\quad x,y,z\in \mathbb{R}^n, t>0.$$
It can be proved that
\begin{equation*}
\label{3.2} \left\|K_{y,t}(x,z)\right\|_{\mathcal
{L}\left(\mathbb{B}, L_{\mathbb{B}}^q\left(\Gamma(0),
\frac{dydt}{t^{n+1}}\right)\right)}\leq C\frac{1}{|x-z|^n},
\end{equation*}
and
\begin{equation*}
\label{3.3}\left\|\nabla_x K_{y,t}(x,z)\right\|_{\mathcal
{L}\left(\mathbb{B},L_\mathbb{B}^q(\Gamma(0),\frac{dydt}{t^{n+1}})\right)}+\left\|\nabla_z
K_{y,t}(x,z)\right\|_{\mathcal
{L}\left(\mathbb{B},L_\mathbb{B}^q(\Gamma(0),\frac{dydt}{t^{n+1}})\right)}
\leq C\frac{1}{|x-z|^{n+1}},
\end{equation*}
for any $x,z \in \mathbb{R}^n, x\neq z.$ We leave the details of the
proof to the reader.
\end{proof}

Together with  Proposition \ref{prop5}, Proposition \ref{lem3} and
Remark \ref{CZ}, we can immediately get the following theorem for
$g_\alpha^q$ and $S_\alpha^q$ with $1<q<\infty$ and
$0<\alpha<\infty.$
\begin{thm}
\label{thm3} Given a Banach space $\mathbb{B},$ $1<q<\infty$ and
$0<\alpha<\infty,$  let $U$ be either the fractional
Littlewood--Paley $g$-function $g_\alpha^q$  or the fractional area
function $S^q_\alpha$, then the following statements are equivalent:
\begin{enumerate}
\item[\textup{(i)}] $U$ maps $L_{c, \mathbb{B}}^\infty(\mathbb{R}^n)$
into $BMO(\mathbb{R}^n).$
\item[\textup{(ii)}] $U$ maps
$H_{\mathbb{B}}^1(\mathbb{R}^n)$ into $L^1(\mathbb{R}^n).$
\item[\textup{(iii)}] $U$ maps $L_{\mathbb{B}}^p(\mathbb{R}^n)$
into $L^p(\mathbb{R}^n)$ for any (or, equivalently, for some) $p\in
(1, \infty).$
\item[\textup{(iv)}] $U$ maps $BMO_{c,
\mathbb{B}}(\mathbb{R}^n)$ into $BMO(\mathbb{R}^n).$
\item[\textup{(v)}] $U$ maps
$L_{\mathbb{B}}^1(\mathbb{R}^n)$ into $L^{1, \infty}(\mathbb{R}^n).$
\end{enumerate}
\end{thm}

\begin{thm}
\label{thm4}  Given a Banach space $\mathbb{B}$ and $2\leq
q<\infty$, the following statements are equivalent:
\begin{enumerate}
\item[\textup{(i)}] $\mathbb{B}$ is of Lusin cotype q.
\item[\textup{(ii)}] For every (or, equivalently, for some) positive integer $n,$
for every (or, equivalently, for some) $p\in(1, \infty)$, and for
every (or, equivalently, for some) $\alpha>0,$ there is a constant
$C>0$ such that
$$\left\|S_{\alpha}^q(f)\right\|_{L^p(\mathbb{R}^n)}\leq
C\|f\|_{L_\mathbb{B}^p(\mathbb{R}^n)},\quad \forall f\in
L_\mathbb{B}^p(\mathbb{R}^n).$$
\end{enumerate}
\end{thm}

\begin{proof}  $\textup{(i)}\Rightarrow \textup{(ii)}.$ By Fubini's theorem,
we have
\begin{multline}\label{3.12}
\left\|S_\alpha^q(f)\right\|_{L^q(\mathbb{R}^n)}^q =
\int_{\mathbb{R}^n}\int_0^\infty \|t^\alpha\partial_t^\alpha
       \mathcal{P}_tf(y)\|_\mathbb{B}^q\left(\int_{\mathbb{R}^n}\chi_{|x-y|<t}dx
       \right)\frac{dydt}{t^{n+1}}\\
   = \int_{\mathbb{R}^n}\int_0^\infty \|t^\alpha\partial_t^\alpha
       \mathcal{P}_tf(y)\|_\mathbb{B}^q\frac{dydt}{t}
   = \|g_\alpha^q(f)\|_{L^q(\mathbb{R}^n)}^q.
\end{multline}
Since $\mathbb{B}$ is of Lusin cotype $q$, by \eqref{3.12} and
Theorem A we get
$$\|S_\alpha^q(f)\|_{L^q(\mathbb{R}^n)}=\|g_{\alpha}^q(f)\|_{L^q(\mathbb{R}^n)}\leq
C\|f\|_{L_\mathbb{B}^q(\mathbb{R}^n)},\quad \forall f\in
L_\mathbb{B}^q(\mathbb{R}^n).$$ Hence, by Theorem \ref{thm3}
$$\|S_\alpha^q(f)\|_{L^p(\mathbb{R}^n)}\leq
C\|f\|_{L_\mathbb{B}^p(\mathbb{R}^n)},\quad \forall f\in
L_\mathbb{B}^p(\mathbb{R}^n), 1< p< \infty.$$

$\textup{(ii)}\Rightarrow \textup{(i)}.$ We only need prove that
there exists a constant $C$ such that
\begin{equation}
\label{3.5} g_\alpha^q(f)(x)\leq CS_\alpha^q(f)(x), \quad \forall
x\in \mathbb{R}^n,
\end{equation}
for a big enough class of nice functions in
$L_\mathbb{B}^p(\mathbb{R}^n).$ Then we have $\displaystyle
\left\|g_\alpha^q(f)\right\|_{L^p(\mathbb{R}^n)}\leq
C\|f\|_{L_\mathbb{B}^p(\mathbb{R}^n)}.$ By Theorem A, $\mathbb{B}$
is of Lusin cotype $q.$

Now, let us prove \eqref{3.5}. We shall follow those ideas in
\cite{Stein2}. It suffices to prove it for $x=0$. Let us denote by
$B(0,t)$ the ball in $\mathbb{R}^{n+1}$ centered at $(0,t)$ and
tangent to the boundary of the cone $\Gamma(0).$ Then the radius of
$B(0,t)$ is $ \frac{\sqrt{2}}{2}t.$ Now the partial derivative
$\partial_t^\alpha \mathcal{P}_tf(x)$ is, like $\mathcal{P}_tf(x)$,
harmonic function. Thus by the mean-value theorem, we have
$$\partial_t^\alpha
\mathcal{P}_tf(0)=\frac{1}{|B(0,t)|}\iint_{B(0,t)}\partial_s^\alpha
\mathcal{P}_sf(x)dxds.$$ By H\"{o}lder's inequality,
\begin{eqnarray*}
\|\partial_t^\alpha \mathcal{P}_tf(0)\|_\mathbb{B}
  &\leq& \frac{1}{|B(0,t)|}\iint_{B(0,t)}\|\partial_s^\alpha
         \mathcal{P}_sf(x)\|_{\mathbb{B}}dxds \nonumber
  \leq \frac{1}{|B(0,t)|^{\frac{1}{q}}}\left(\iint_{B(0,t)}\|\partial_s^\alpha
         \mathcal{P}_sf(x)\|_{\mathbb{B}}^qdxds\right)^{\frac{1}{q}}.
\end{eqnarray*}
Integrating this inequality, we obtain
\begin{multline*}
\int_0^\infty t^{\alpha q}\|\partial_t^\alpha
\mathcal{P}_tf(0)\|_\mathbb{B}^q\frac{dt}{t}
  \leq C\int_0^\infty t^{\alpha q-n-2}\iint_{B(0,t)}\|\partial_s^\alpha
         \mathcal{P}_sf(x)\|_{\mathbb{B}}^qdxdsdt \nonumber\\
  \leq C\iint_{\Gamma(0)}\left(\int_{c_1s}^{c_2s}t^{\alpha q-n-2}dt\right)\|\partial_s^\alpha
         \mathcal{P}_sf(x)\|_{\mathbb{B}}^qdxds \nonumber
  \leq C\iint_{\Gamma(0)}\left\|s^\alpha \partial_s^\alpha
         \mathcal{P}_sf(x)\right\|_{\mathbb{B}}^q\frac{dxds}{s^{n+1}}
\end{multline*}
by using Fubini's theorem and $(x,s)\in B(0,t)$ implying $c_1s\leq
t\leq c_2s$, for two positive constants $c_1$ and $c_2$. Hence, we
get inequality \eqref{3.5}.
\end{proof}

\begin{thm}
\label{thm5} Given a Banach space $\mathbb{B}$ and $1<q\leq 2,$ the
following statements are equivalent:
\begin{enumerate}
\item[\textup{(i)}] $\mathbb{B}$ is of Lusin type q.
\item[\textup{(ii)}] For every (or, equivalently, for some) positive integer $n$, for
every (or, equivalently, for some) $p\in(1, \infty)$, and for every
(or, equivalently, for some) $\alpha>0,$ there is a constant $C>0$
such that
\begin{equation*}
\|f\|_{L_\mathbb{B}^p(\mathbb{R}^n)}\leq C
\left\|S_{\alpha}^q(f)\right\|_{L^p(\mathbb{R}^n)},\quad \forall
f\in L_\mathbb{B}^p(\mathbb{R}^n).
\end{equation*}
\end{enumerate}
\end{thm}

\begin{proof} $\textup{(i)}\Rightarrow \textup{(ii)}.$ Since $\mathbb{B}$ is
of Lusin type $q,$ by Theorem B and \eqref{3.5} we have
\begin{eqnarray*}
\|f\|_{L_\mathbb{B}^p(\mathbb{R}^n)}
   \leq C\left\|g_{\alpha}^q(f)\right\|_{L^p(\mathbb{R}^n)}
   \leq C\left\|S_{\alpha}^q(f)\right\|_{L^p(\mathbb{R}^n)}.
\end{eqnarray*}

$\textup{(ii)}\Rightarrow \textup{(i)}.$  We shall prove
$\|S_\alpha^{q'} (g) \|_ {L^{p'}(\mathbb{R}^n)} \le C \| g
\|_{L^{p'}_{\mathbb{B}^*}(\mathbb{R}^n)}.$ We can choose $b \in
L^p_{L^q_{\mathbb{B}}\left(\Gamma(0),\frac{dzdt}{t^{n+1}}\right)}(\mathbb{R}^n)$
of unit norm such that
\begin{eqnarray*}
\|S_\alpha^{q'} (g) \|_ {L^{p'}(\mathbb{R}^n)} &=&  \| t^\alpha
\partial_t^\alpha \mathcal{P}_t g(y-z)
\|_{L^{p'}_{L^{q'}_{\mathbb{B}}\left(\Gamma(0),\frac{dzdt}{t^{n+1}}\right)}(\mathbb{R}^n)}
\\ &= &
\int_{\mathbb{R}^n} \int_{\Gamma(0)}\langle t^\alpha \partial_t^\alpha \mathcal{P}_t g(y-z),   b(y,z,t) \rangle \frac{dz dt }{t^{n+1}} dy  \\
&=& \int_{\mathbb{R}^n} \int_{\Gamma(0)} \Big\langle
\int_{\mathbb{R}^n}t^\alpha\partial_t^\alpha \mathcal{P}_t
(y-z-\tilde{z})\,
 g(\tilde{z})\, d\tilde{z} ,  b(y,z,t) \Big\rangle \frac{dz dt }{t^{n+1}} dy \\
&=& \int_{\mathbb{R}^n} \Big\langle g(\tilde{z}),
\int_{\Gamma(0)}\int_{\mathbb{R}^n} t^\alpha \partial_t^\alpha
\mathcal{P}_t (y-z-\tilde{z}) b(y,z,t)  dy  \frac{dz dt }{t^{n+1}}  \Big\rangle\,  d\tilde{z} \\
&\le & \|g \|_{L^{p'}_{\mathbb{B}^*}(\mathbb{R}^n)} \|
G(b)\|_{L^{p}_{\mathbb{B}}(\mathbb{R}^n)} \le \|g
\|_{L^{p'}_{\mathbb{B}^*}(\mathbb{R}^n)} \|S_\alpha^q
(G(b))\|_{L^{p}_{\mathbb{B}}(\mathbb{R}^n)},
\end{eqnarray*}
where $\displaystyle G(b) (\tilde{z}) = \int_{\Gamma(0)}
\int_{\mathbb{R}^n} t^\alpha \partial_t^\alpha \mathcal{P}_t
(y-z-\tilde{z}) b(y,z,t)  dy  \frac{dz dt }{t^{n+1}} $ and in the
last inequality we used the hypothesis. Let us observe that we will
have proved the result as soon as we prove $\|S_\alpha^q
(G(b))\|_{L^{p}_{\mathbb{B}}(\mathbb{R}^n)}  \le C\|b\|_{L^p_{
L^q_\mathbb{B}\left(\Gamma(0), \frac{dz
dt}{t^{n+1}}\right)}(\mathbb{R}^n)}$. We shall prove this by
following a parallel argument to the proof of $(ii) \implies (i)$ in
Theorem B and we also borrow the ideal from \cite{OX}. Observe that
\begin{align*}
&S_\alpha^q( G(b) ) (x) \\&=\Big(\int_{\Gamma(0)} \Big\| s^\alpha
\partial_s^\alpha \mathcal{P}_s \Big( \int_{\Gamma(0)}   \int_{\mathbb{R}^n}
t^\alpha \partial_t^\alpha \mathcal{P}_t (y-z-\cdot)
b(y,z,t)  dy  \frac{dz dt }{t^{n+1}} \Big)( x+u   ) \Big\|_\mathbb{B}^q \frac{du ds}{s^{n+1}}\Big)^{1/q} \\
&= \Big(\int_{\Gamma(0)} \Big\| s^\alpha \partial_s^\alpha \mathcal{P}_s \Big(
\int_{\Gamma(0)}
  \int_{\mathbb{R}^n} t^\alpha \partial_t^\alpha \mathcal{P}_t (-y+z+\cdot) b(y,z,t)  dy
   \frac{dz dt }{t^{n+1}} \Big)(x+u) \Big\|_\mathbb{B}^q \frac{du ds}{s^{n+1}}\Big)^{1/q} \\
&= \Big(\int_{\Gamma(0)} \Big\|\Big( \int_{\Gamma(0)}
\int_{\mathbb{R}^n} s^\alpha \partial_s^\alpha \mathcal{P}_s
t^\alpha \partial_t^\alpha \mathcal{P}_t (-y+z+x+u) b(y,z,t)dy\frac{dz dt }{t^{n+1}}
 \Big) \Big\|_\mathbb{B}^q \frac{du ds}{s^{n+1}}\Big)^{1/q} \\
 &= \Big(\int_{\Gamma(0)} \Big\|
 \Big( \int_{\Gamma(0)}   \int_{\mathbb{R}^n} s^\alpha t^\alpha \partial_u^{2\alpha}
 \mathcal{P}_u\big|_{u=s+t}  (-y+z+x+u) b(y,z,t) dy  \frac{dz dt }{t^{n+1}} \Big) \Big\|_\mathbb{B}^q \frac{du ds}{s^{n+1}}\Big)^{1/q}.
\end{align*}
It is an easy exercise to  prove that
$$| s^\alpha t^\alpha \partial_u^{2\alpha} \mathcal{P}_u|_{u=s+t} | \le C \frac{s^\alpha t^\alpha}{ (s+t+|x|)^{n+2\alpha}}.$$
In this circumstances, it can be proved that the operator
$$ b \longrightarrow \mathcal{U}(b)(x,u,s)=\int_{\mathbb{R}^n}\int_{\Gamma(0)}
s^\alpha t^\alpha \partial_u^{2\alpha} \mathcal{P}_u|_{u=s+t}
(-y+z+x+u) b(y,z,t)\frac{dz dt }{t^{n+1}} dy $$ can be handled by
using Calder\'on--Zygmund techniques and $\mathcal{U}$ is bounded on
$L^p_{L^q_\mathbb{B}(\Gamma(0),
\frac{duds}{s^{n+1}})}(\mathbb{R}^n)$ for every $1<p,q < \infty$ and
every Banach space $\mathbb{B}$, see the details in \cite[Section
2]{OX}. The proof of the theorem ends by observing that
$\displaystyle S_\alpha^q( G(b) )  = \|
\mathcal{U}(b)\|_{L^q_\mathbb{B}(\Gamma(0), \frac{duds}{s^{n+1}})}$.
\end{proof}

Now, let us consider the relationship between the geometry
properties of the Banach space $\mathbb{B}$ and the fractional
$g_\lambda^*$-function $g_{\lambda, \alpha}^{q,*}$.

\begin{thm}\label{thm7}  Given a Banach space $\mathbb{B}$, $2\leq q<\infty$ and $\lambda>1$, the following statements are
equivalent:
\begin{enumerate}
\item[\textup{(i)}] $\mathbb{B}$ is of Lusin cotype $q$.
\item[\textup{(ii)}] For every (or, equivalently, for some) positive integer
$n,$ for every (or, equivalently, for some) $p\in [q, \infty)$, and
for every (or, equivalently, for some) $\alpha>0,$ there is a
constant $C>0$ such that
$$\left\|g_{\lambda, \alpha}^{q,
*}(f)\right\|_{L^p(\mathbb{R}^n)}\leq
C\|f\|_{L_\mathbb{B}^p(\mathbb{R}^n)},\quad \forall f\in
L_\mathbb{B}^p(\mathbb{R}^n).$$
\end{enumerate}
\end{thm}

\begin{proof}
 $\textup{(i)}\Rightarrow \textup{(ii)}.$ Since
$\lambda>1,$ the function $(1+|x|)^{-\lambda n}$ is integrable and
hence for good enough function $h(x)\geq 0,$  we have
\begin{equation}\label{02.20}    
\sup\limits_{t>0} \int_{\mathbb{R}^n}
\frac{1}{t^n}\Big(\frac{t}{t+|x-y|}\Big)^{\lambda n}h(y)dy\leq
CMh(x),
\end{equation}
where $Mh$ is the Hardy--Littlewood maximal function of $h$. By
\eqref{02.20} and  H\"{o}lder's inequality, we have
\begin{align*}
\int_{\mathbb{R}^n}\left(g_{\lambda, \alpha}^{q, *}(f)(x)\right)^q
h(x)dx
   &= \int_{\mathbb{R}^n}\int_{\mathbb{R}^n}\int_0^\infty
       \left\|t^\alpha\partial_t^{\alpha} \mathcal{P}_tf(y)\right\|_{\mathbb{B}}^q
       \Big(\frac{t}{t+|x-y|}\Big)^{\lambda
       n}\frac{dt}{t^{n+1}}dyh(x)dx\\
   &\leq C\int_{\mathbb{R}^n}\left(g_{\alpha}^{q}(f)(y)\right)^qMh(y)dy
   \leq C\left\|g_\alpha^{q}(f)\right\|_{L^p{\left(\mathbb{R}^n\right)}}^q
         \left\|Mh\right\|_{L^{\frac{p}{p-q}}\left(\mathbb{R}^n\right)}.
\end{align*}
Here, when $p=q$, let
$L^{\frac{p}{p-q}}\left(\mathbb{R}^n\right)=L^{\infty}\left(\mathbb{R}^n\right).$
Since $M$ is bounded on $L^r\left(\mathbb{R}^n\right) (1<r\leq
\infty)$, we get
\begin{equation*}
\int_{\mathbb{R}^n}\left(g_{\lambda,\alpha}^{q,*}(f)(x)\right)^q
          h(x)dx
   \leq C\left\|g_\alpha^{q}(f)\right\|_{L^p{\left(\mathbb{R}^n\right)}}^q
         \left\|h\right\|_{L^{\frac{p}{p-q}}\left(\mathbb{R}^n\right)}.
\end{equation*}
Taking supremum over all $h$ in
$L^{\frac{p}{p-q}}\left(\mathbb{R}^n\right),$ we get
\begin{equation}\label{02.21}
\Big\|g_{\lambda,\alpha}^{q,*}(f)\Big\|_{L^p
\left(\mathbb{R}^n\right)}\leq C\left\|g_\alpha^q(f)\right\|_{L^p
\left(\mathbb{R}^n\right)},  \quad  q \le p.
\end{equation}
Since $\mathbb{B}$ is of Lusin cotype $q$, by Theorem A  and
\eqref{02.21} we get $ \displaystyle \|g_{\lambda, \alpha}^{q,
*}(f)\|_{L^p(\mathbb{R}^n)}\leq
C\|f\|_{L_\mathbb{B}^p(\mathbb{R}^n)}.$

$\textup{(ii)}\Rightarrow \textup{(i)}.$ On the domain
$\Gamma(x)=\set{(y, t)\in \mathbb{R}_+^n: |y-x|<t}$, we have
$$\left(\frac{t}{|x-y|+t}\right)^{\lambda n}>\left(\frac{1}{2}\right)^{\lambda n}.$$
Hence
\begin{align}\label{3.8}
S_\alpha^q(f)(x)
&=\left(\iint_{\Gamma(x)}\left\|t^\alpha\partial_t^\alpha\mathcal
                 {P}_tf(y)\right\|_{\mathbb{B}}^q\frac{dydt}{t^{n+1}}\right)^{\frac{1}{q}}\nonumber\\
          &\le \left(\iint_{\Gamma(x)}2^{\lambda n}\left(\frac{t}{|x-y|+t}\right)^{\lambda n}
                  \|t^\alpha\partial_t^\alpha\mathcal{P}_tf(y)\|_{\mathbb{B}}^q\frac{dydt}
                  {t^{n+1}}\right)^{\frac{1}{q}}\nonumber\\
         &\le 2^{\frac{\lambda n}{q}}\left(\iint_{\mathbb{R}_+^{n+1}}\left(\frac{t}{|x-y|+t}\right)^{\lambda n}
                 \left\|t^\alpha\partial_t^\alpha\mathcal{P}_tf(y)\right\|_{\mathbb{B}}^q
                 \frac{dydt}{t^{n+1}}\right)^{\frac{1}{q}}\\
         &= 2^{\frac{\lambda n}{q}}g_\lambda^{q,*}(f)(x),\quad \forall x\in
         \mathbb{R}^n.\nonumber
\end{align}
Hence $\displaystyle
\left\|S_\alpha^q(f)\right\|_{L_{\mathbb{B}}^p(\mathbb{R}^n)}\leq
2^{\frac{\lambda n}{q}}\left\|g_{\lambda, \alpha}^{q,
*}(f)\right\|_{L_{\mathbb{B}}^p(\mathbb{R}^n)}\leq C\|f\|_{L_\mathbb{B}^p(\mathbb{R}^n)},
$ for any $f\in L_{\mathbb{B}}^p(\mathbb{R}^n).$ Then, by  Theorem
\ref{thm4}, $\mathbb{B}$ is of Lusin cotype $q.$
\end{proof}

\begin{thm}\label{thm9}  Given a Banach space $\mathbb{B}$, $1<q\leq 2$
 and $\lambda>1$, the following statements are
equivalent:
\begin{enumerate}
\item[\textup{(i)}] $\mathbb{B}$ is of Lusin type $q$.
\item[\textup{(ii)}] For every (or, equivalently, for some) positive integer
$n,$ for every (or, equivalently, for some) $p\in [q, \infty)$, and
for every (or, equivalently, for some) $\alpha>0,$ there is a
constant $C>0$ such that
$$\|f\|_{L_\mathbb{B}^p(\mathbb{R}^n)}\leq
C\left\|g_{\lambda, \alpha}^{q,*}(f)\right\|_{L^p(\mathbb{R}^n)},
\quad \forall f\in L_\mathbb{B}^p(\mathbb{R}^n).$$
\end{enumerate}
\end{thm}

\begin{proof}
$\textup{(i)}\Rightarrow \textup{(ii)}.$ Since $\mathbb{B}$ is of
Lusin type $q$, by Theorem \ref{thm5} and \eqref{3.8} we get
 $$
\|f\|_{L_\mathbb{B}^p(\mathbb{R}^n)} \leq
C\left\|S_{\alpha}^q(f)\right\|_{L^p(\mathbb{R}^n)} \leq
C\left\|g_{\lambda, \alpha}^{q,*}(f)\right\|_{L^p(\mathbb{R}^n)},
\quad \forall f\in L_\mathbb{B}^p(\mathbb{R}^n).
$$
$\textup{(ii)}\Rightarrow \textup{(i)}.$ By \eqref{02.21}, we get
\begin{eqnarray*}
\|f\|_{L_\mathbb{B}^p(\mathbb{R}^n)}
   \leq C \left\|g_{\lambda,\alpha}^{q, *}(f)\right\|_{L^p(\mathbb{R}^n)}
   \leq C \left\|g_\alpha^q(f)\right\|_{L^p
\left(\mathbb{R}^n\right)}, \quad \forall f\in
L_\mathbb{B}^p(\mathbb{R}^n).
\end{eqnarray*}
Then by Theorem B, $\mathbb{B}$ is of Lusin type $q$.
\end{proof}

\section{Proof of Theorem C} \label{A.E.}

\begin{proof}[Proof of Theorem C]
By Theorem A, Theorem \ref{thm3}, Theorem \ref{thm4} and Theorem
\ref{thm7}, we have (i) $\Rightarrow$ (ii), (i) $\Rightarrow$ (iii)
and (i) $\Rightarrow$ (iv). \\
Let us prove the converse. $\textup{(ii)}\Rightarrow \textup{(i)}.$
Let $p_0\in(1, \infty).$ Observe that
\begin{multline*}
g_\alpha^q(f)(x)
   =\left(\int_0^\infty \left\|t^\alpha\partial_t^\alpha
       \mathcal{P}_tf(x)\right\|_{\mathbb{B}}^q\frac{dt}{t}\right)^{\frac{1}{q}}\\
         = \sup\limits_{j\in \mathbb{Z}^+}\left(\int_{\frac{1}{j}}^{j}
        \left\|t^\alpha\partial_t^\alpha
       \mathcal{P}_tf(x)\right\|_{\mathbb{B}}^q\frac{dt}{t}\right)^{\frac{1}{q}}
   =\sup\limits_{j\in \mathbb{Z}^+}\left\|T^{j}(f)(x,t)\right\|_{L_{\mathbb{B}}
      ^q(\mathbb{R}_+, \frac{dt}{t})},
\end{multline*}
where $T^{j}(f)(x,t)=t^\alpha\partial_t^\alpha
\mathcal{P}_tf(x)\chi_{\{\frac{1}{j}<t<j\}}$ is the operator which
sends $\mathbb{B}$-valued functions to
$L_{\mathbb{B}}^q(\mathbb{R}_+, \frac{dt}{t})$-valued functions. It
is clear that $T^{j}$ is bounded from
$L_\mathbb{B}^{p_0}(\mathbb{R}^n)$ to
$L_{L_{\mathbb{B}}^q(\mathbb{R}_+,  
\frac{dt}{t})}^{p_0}(\mathbb{R}^n).$  Let
$T_N^{j}(f)(x)=T^{j}(f)(x)\chi_{B_N}(x),$ where $B_N=B(0, N)$ is the
ball in $\mathbb{R}^n,$ for any $N>0.$ So $T_N^{j}$ is bounded from
$L_\mathbb{B}^{p_0}(\mathbb{R}^n)$ to
$L_{L_{\mathbb{B}}^q(\mathbb{R}_+, \frac{dt}{t})}^{p_0}(B_N).$ Then
we have
\begin{multline}\label{6.1}
\Big|\Big\{x\in B_N:
\left\|T_N^{j}(f)(x)\right\|_{L_{\mathbb{B}}^q(\mathbb{R}_+,
\frac{dt}{t})}>\lambda
\|f\|_{L_\mathbb{B}^{p_0}(\mathbb{R}^n)}\Big\}\Big|\\
 \leq \frac{1}{\lambda^{p_0}\|f\|_{L_\mathbb{B}^{p_0}(\mathbb{R}^n)}^{p_0}}
         \int_{B_N}\left\|T_N^{j}(f)(x)\right\|_{L_{\mathbb{B}}^q(\mathbb{R}_+,
\frac{dt}{t})}^{p_0}dx
 \leq \frac{C}{\lambda^{p_0}}.
\end{multline}

Let $\displaystyle \mathcal{M}=\Big\{f: f\ \textup{is
${L_{\mathbb{B}}^q(\mathbb{R}_+, \frac{dt}{t})}$-valued and strong
measurable on} \ B_N\Big\}.$ In the finite measurable space,
$\left(B_N, \mathcal{M}\right)$, we introduce the following topology
basis. For any $\varepsilon>0,$ let \newline $$\displaystyle V_{B_N,
\varepsilon}=\Big\{f\in \mathcal{M}: \Big|\Big\{x\in B_N:
\Big\|f(x)\Big\|_{L_{\mathbb{B}}^q(\mathbb{R}_+,
\frac{dt}{t})}>\varepsilon\Big\}\Big|<\varepsilon\Big\}.$$ We denote
the topology space  on $B_N$ by $L^0_{L_{\mathbb{B}}^q(\mathbb{R}_+,
\frac{dt}{t})}(B_N)$. By \eqref{6.1}, we have $$\displaystyle
\lim_{\lambda \rightarrow \infty}\Big|\Big\{x\in B_N:
\Big\|T_N^{j}(f)(x)\Big\|_{L_{\mathbb{B}}^q(\mathbb{R}_+,
\frac{dt}{t})} >\lambda
\|f\|_{L_\mathbb{B}^{p_0}(\mathbb{R}^n)}\Big\}\Big|=0.$$ So for any
$\varepsilon>0,$ there exists $\lambda_\varepsilon>0$ such that
$$\Big|\Big\{x\in B_N: \left\|T_N^{j}(f)(x)\right\|_{L_{\mathbb{B}}^q(\mathbb{R}_+,
\frac{dt}{t})}>\lambda
\|f\|_{L_\mathbb{B}^{p_0}(\mathbb{R}^n)}\Big\}\Big|<\varepsilon,
\quad \lambda\ge\lambda_\varepsilon.$$ Then for $\varepsilon$ given
above, there exists a constant $\displaystyle
\delta_\varepsilon=\frac{\varepsilon}{\lambda_\varepsilon}$, such
that for any $\displaystyle
\norm{f}_{L_\mathbb{B}^{p_0}(\mathbb{R}^n)}<\delta_\varepsilon$ we
have
\begin{equation*}
\Big|\Big\{x\in B_N:
\Big\|T_N^{j}(f)(x)\Big\|_{L_{\mathbb{B}}^q(\mathbb{R}_+,
\frac{dt}{t})}> \varepsilon\Big\}\Big|\le\Big|\Big\{x\in B_N:
\Big\|T_N^{j}(f)(x)\Big\|_{L_{\mathbb{B}}^q(\mathbb{R}_+,
\frac{dt}{t})}> \lambda_\varepsilon
\|f\|_{L_\mathbb{B}^{p_0}(\mathbb{R}^n)}\Big\}\Big|<\varepsilon.
\end{equation*}
This means that $\displaystyle T_N^j(f)\in V_{B_N, \varepsilon}$ for
any $f\in L_\mathbb{B}^{p_0}(\mathbb{R}^n)$ with $\displaystyle
\norm{f}_{L_\mathbb{B}^{p_0}(\mathbb{R}^n)}<\delta_\varepsilon.$
Hence $T_N^{j}$ is continuous from $\displaystyle
L_\mathbb{B}^{p_0}(\mathbb{R}^n)$ to $\displaystyle
L_{L_{\mathbb{B}}^q(\mathbb{R}_+, \frac{dt}{t})}^0(B_N).$ Let
$\displaystyle U_N=\left\{T_N^j(f)\right\}_{j=1}^\infty.$ Since
$g_\alpha^q(f)(x) < \infty $ a.e.,  $U_N$ is a well defined linear
operator from $\displaystyle L_\mathbb{B}^{p_0}(\mathbb{R}^n)$ to
$L_{\ell^\infty({L_{\mathbb{B}}^q(\mathbb{R}_+,
\frac{dt}{t})})}^0(B_N)$.  As $B_N$ has finite measure, the space
$L_{\ell^\infty({L_{\mathbb{B}}^q(\mathbb{R}_+,
\frac{dt}{t})})}^0(B_N)$ is metrizable and  complete. Then by the
closed graph theorem, the operator $U_N$ is continuous. As
$\displaystyle g_{\alpha,
N}^q(f)(x)=\left\|T_N^j(f)(x)\right\|_{\ell^\infty({L_{\mathbb{B}}^q(\mathbb{R}_+,
\frac{dt}{t})})}$, we get that $g_{\alpha, N}^q$ is continuous from
$\displaystyle L_\mathbb{B}^{p_0}(\mathbb{R}^n)$ to $L^0(B_N).$
Therefore  for any $\varepsilon>0,$ there exists
$\delta_\varepsilon>0$ such that
$$\abs{\left\{x\in B_N: \abs{g_\alpha^q(h)(x) }>\varepsilon\right\}}<\varepsilon, \,
\,\hbox{ for } \,  \norm{h}_{L_\mathbb{B}^{p_0}(\mathbb{R}^n)}<\delta_\varepsilon.$$
In particular, for any $0<r<\varepsilon,$ there exists $\delta_r>0$
such that
$$\abs{\left\{x\in
B_N:\abs{g_\alpha^q(h)}>r\right\}}<\varepsilon, \, \, \hbox{ for }
\norm{h}_{L_\mathbb{B}^{p_0}(\mathbb{R}^n)}<\delta_r.$$  Now let $g$
be an element of $L_\mathbb{B}^{p_0}(\mathbb{R}^n)$ with
$\norm{g}_{L_\mathbb{B}^{p_0}(\mathbb{R}^n)}\neq 0$ and
$\displaystyle
h=\frac{g}{\norm{g}_{L_\mathbb{B}^{p_0}(\mathbb{R}^n)}}\frac{\delta_r}{2}.$
Then we have $\displaystyle
\norm{h}_{L_\mathbb{B}^{p_0}(\mathbb{R}^n)}<\frac{\delta_r}{2}$ and
\begin{align*} \varepsilon>\abs{\big\{x\in
B_N:\abs{g_\alpha^q(h)}>r\big\}}>\abs{\Big\{x\in B_N:
\abs{g_\alpha^q(h)}>\varepsilon\Big\}}=\Big|\Big\{x\in B_N:
\abs{g_\alpha^q(g)}>\frac{2\varepsilon
\norm{g}_{L_\mathbb{B}^{p_0}(\mathbb{R}^n)}}{\delta_r}\Big\}\Big|.
\end{align*}
Let $\displaystyle \mu_\varepsilon =\frac{2\varepsilon}{\delta_r}.$
Then when $\mu\ge \mu_\varepsilon$, we have
\begin{align}\label{cont} \abs{\left\{x\in B_N:
\abs{g_\alpha^q(g)}>\mu\norm{g}_{L_\mathbb{B}^{p_0}(\mathbb{R}^n)}\right\}}
\le \Big|\Big\{x\in B_N: \abs{g_\alpha^q(g)}>\frac{2\varepsilon
\norm{g}_{L_\mathbb{B}^{p_0}(\mathbb{R}^n)}}{\delta_r}\Big\}\Big|<\varepsilon.
\end{align}

Let $\displaystyle f\in L_\mathbb{B}^1(\mathbb{R}^n)$ and
$\lambda>0,$ we perform the Calder\'{o}n--Zygmund decomposition as
the sum $f=g+b$ such that $\displaystyle
\left\|g\right\|_{L_\mathbb{B}^1(\mathbb{R}^n)}\leq
\left\|f\right\|_{L_\mathbb{B}^1(\mathbb{R}^n)}$ and
$\left\|g\right\|_{L_\mathbb{B}^\infty(\mathbb{R}^n)}\leq 2\lambda$.
Then we have
\begin{equation}\label{04.1}
\left\|g\right\|_{L_\mathbb{B}^{p_0}(\mathbb{R}^n)}\leq
(2\lambda)^{\frac{p_0-1}{p_0}}\left\|f\right\|_{L_\mathbb{B}^1(\mathbb{R}^n)}^{\frac{1}{p_0}}
\end{equation}
and
\begin{equation}\label{04.2}
\Big|\Big\{x\in \mathbb{R}^n: \big|g_{\alpha}^q(b)(x)\big|
>\frac{\lambda}{2}\Big\}\Big|\leq
\frac{C}{\lambda}\left\|f\right\|_{L_\mathbb{B}^1(\mathbb{R}^n)}.
\end{equation}
Indeed, \eqref{04.1} is trivial from the estimates of $g.$ For
\eqref{04.2},  we observe that by Proposition \ref{prop5},
$g_\alpha^q(f)$ can be expressed as an
$L^q_\mathbb{B}(\mathbb{R}_+,\frac{dt}{t})$-norm of a
Calder\'{o}n--Zygmund operator with a regular kernel. In these
circumstances, it can be observed that the boundedness of the
measure of the set appearing in \eqref{04.2} depends only on the
kernel of the operator and not on the boundedness of the operator,
see \cite{duo}. Therefore, by \eqref{04.1} and \eqref{04.2}, we have
\begin{align*}
&\Big|\Big\{x\in B_N:
\Big|g_{\alpha}^q(f)(x)\Big|>\lambda\Big\}\Big|
   \leq\Big|\Big\{x\in B_N: \Big|g_{\alpha}^q(g)(x)\Big|
           >\frac{\lambda}{2}\Big\}\Big|+\Big|\Big\{x\in \mathbb{R}^n:
           \Big|g_{\alpha}^q(b)(x)\Big|>\frac{\lambda}{2}\Big\}\Big|\nonumber\\
   &=\Big|\Big\{x\in B_N: \Big|g_{\alpha}^q(g)(x)\Big|
           >\frac{\lambda}{2\left\|g\right\|_{L_\mathbb{B}^{p_0}(\mathbb{R}^n)}}\left\|
           g\right\|_{L_\mathbb{B}^{p_0}(\mathbb{R}^n)}\Big\}\Big| \nonumber\\
   &\quad     +\Big|\Big\{x\in \mathbb{R}^n:
           \left|g_{\alpha}^q(b)(x)\right|>\frac{\lambda}{2}\Big\}\Big|\\
   &\leq \Big|\Big\{x\in B_N: \left|g_{\alpha}^q(g)(x)\right|
           >\frac{\lambda^{\frac{1}{{p_0}}}}{2^{2-\frac{1}{p_0}}\left\|f\right\|_{L_\mathbb{B}^1(\mathbb{R}^n)}^{\frac{1}{p_0}}}\left\|
           g\right\|_{L_\mathbb{B}^{p_0}(\mathbb{R}^n)}\Big\}\Big|
        +\frac{C}{\lambda}\left\|f\right\|_{L_\mathbb{B}^1(\mathbb{R}^n)}\nonumber \\
   &= \Big|\Big\{x\in B_N: \Big|g_{\alpha,N}^q(g)(x)\Big|
           >\frac{\lambda^{\frac{1}{{p_0}}}}{2^{2-\frac{1}{p_0}}\left\|f\right\|_{L_\mathbb{B}^1(\mathbb{R}^n)}^{\frac{1}{{p_0}}}}\left\|
           g\right\|_{L_\mathbb{B}^{p_0}(\mathbb{R}^n)}\Big\}\Big|
        +\frac{C}{\lambda}\left\|f\right\|_{L_\mathbb{B}^1(\mathbb{R}^n)}.\nonumber
\end{align*}
Now, given $\varepsilon >0$ we perform the Calder\'on-Zygmund
decomposition with $\lambda$ such that  $\displaystyle
\lambda^{\frac{1}{p_0}}>2^{2-\frac{1}{p_0}}\left\|f\right\|_{L_\mathbb{B}^1(\mathbb{R}^n)}^{\frac{1}{{p_0}}}
\mu_\varepsilon$. Then, by \eqref{cont}, we have
\begin{align*}
\left|\left\{x\in B_N:
\left|g_{\alpha}^q(f)(x)\right|>\lambda\right\}\right|
   &\leq\left|\left\{x\in B_N: \left|g_{\alpha}^q(g)(x)\right|
           >\mu_\varepsilon \|g\|_{ L_\mathbb{B}^{p_0}(\mathbb{R}^n)}
            \right\}\right|+\frac{C}{\lambda}\left\|f\right\|_{L_\mathbb{B}^1(\mathbb{R}^n)}\nonumber\\
   &\le \varepsilon
        +\frac{C}{\lambda}\left\|f\right\|_{L_\mathbb{B}^1(\mathbb{R}^n)}.\nonumber
\end{align*}
This clearly implies $g_\alpha^q(f)(x)<\infty$ a.e. $x\in
\mathbb{R}^n,$ for any $f\in L_\mathbb{B}^1(\mathbb{R}^n).$ We apply
Theorem \ref{tipoa.e.} and  get the result.

To prove that $\textup{(iii)}\Rightarrow \textup{(i)}$, we can use
the same argument as above but with a very small modification. We
only need note that
\begin{multline*} S_\alpha^q(f)(x)
   =\left(\iint_{\Gamma(x)} \left\|t^\alpha\partial_t^\alpha
       \mathcal{P}_tf(y)\right\|_{\mathbb{B}}^q\frac{dydt}{t^{n+1}}\right)^{\frac{1}{q}}
   =\left(\int_0^\infty\int_{\abs{y-x}<t} \left\|t^\alpha\partial_t^\alpha
       \mathcal{P}_tf(y)\right\|_{\mathbb{B}}^q\frac{dy}{t^n}\frac{dt}{t}\right)^{\frac{1}{q}}\\
    = \sup_{j\in \mathbb{Z}^+}\left(\int_{\frac{1}{j}}^{j}\int_{\abs{y-x}<t}
        \left\|t^\alpha\partial_t^\alpha
       \mathcal{P}_tf(y)\right\|_{\mathbb{B}}^q\frac{dy}{t^n}\frac{dt}{t}\right)^{\frac{1}{q}}
   =\sup_{j\in \mathbb{Z}^+}\left\|T^{j}(f)(x,t)\right\|_{L_{\mathbb{B}}
      ^q(\mathbb{R}_+, \frac{dt}{t})},
\end{multline*}
where $\displaystyle
T^{j}(f)(x,t)=\int_{\abs{y-x}<t}\norm{t^\alpha\partial_t^\alpha
\mathcal{P}_tf(y)}_{\mathbb{B}}^q\frac{dy}{t^n}\chi_{\{\frac{1}{j}<t<j\}}$
is the operator which sends $\mathbb{B}$-valued functions to
$L^q(\mathbb{R}_+, \frac{dt}{t})$-valued functions. And $T^{j}$ is
bounded from $L_\mathbb{B}^{p_0}(\mathbb{R}^n)$ to
$L_{L^q(\mathbb{R}_+,  
\frac{dt}{t})}^{p_0}(\mathbb{R}^n),\ 1<p_0<\infty$ also. Now we can
continue the proof as in the case of $g_\alpha^q.$

$\textup{(iv)}\Rightarrow \textup{(i)}.$ Assuming that
$\displaystyle g_{\lambda, \alpha}^{q, *}(f)(x)<\infty$
a.e.~$x\in\mathbb{R}^n,$ by \eqref{3.8} we know that $\displaystyle
S_{\alpha}^{q}(f)(x)\leq Cg_{\lambda, \alpha}^{q, *}(f)(x)<\infty$
a.e.~$x\in\mathbb{R}^n.$ Then by $\textup{(iii)}\Rightarrow
\textup{(i)}$,  $\mathbb{B}$ is of Lusin cotype $q$.
\end{proof}

\section{UMD Spaces}\label{UMD}

Now we give the proof of Theorem D.  Clearly it is enough to prove
$\textup{(ii)}\Rightarrow \textup{(i)}$. Let  $1<p_0<\infty$ and
assume that $\displaystyle \lim_{\varepsilon \rightarrow 0^+
}\int_{|x-y|> \varepsilon} \frac{f(y)}{x-y}dy $ exists a.e.~$x\in
\mathbb{R}$ for any $f\in L^{p_0}_\mathbb{B}(\mathbb{R}).$ Then the
maximal operator
 $\displaystyle H^*f(x) = \sup_{\varepsilon>0} \Big\|\int_{|x-y|> \varepsilon} \frac{f(y)}{x-y}dy\Big\|_{\mathbb{B}}$
is finite a.e. $x\in \mathbb{R}$. Our idea is to apply the method
developed in the proof of
   $(ii) \Rightarrow (i)$ of Theorem C. However, we cannot apply it directly since $ H^*$ can't be expressed
    as a  norm of a Calder\'on--Zygmund operator with a regular  kernel.  Let $\varphi$ be a smooth function
     such that $\chi_{[\frac{3}{2},\infty)} \le \varphi \le  \chi_{[\frac{1}{2},\infty)}$. Consider the operator
      $\displaystyle H_\varphi^*f(x)  = \sup_{\varepsilon>0}  \Big\|\int_{\mathbb{R }}
      \varphi\Big(\frac{|x-y|}{\varepsilon}\Big) f(y) dy\Big\|_{\mathbb{B}}$.  It can be easily checked that
\begin{equation} \label{equiv}
\abs{ H_\varphi^*f(x) -  H^*f(x)} \le C M(\norm{f}_{\mathbb{B}})(x)
,\quad \hbox{a.e.}\  x \in \mathbb{R},
\end{equation}
where $M$ denotes the Hardy--Littlewood maximal function. Therefore,
the operator $H_\varphi^*f(x) < \infty,\ \hbox{a.e.}\ x \in
\mathbb{R}.$ Observe that this operator can be expressed  as
$$ H_\varphi^*f(x) =  \Big\| \Big\{  \int_{\mathbb{R }} \varphi
\Big(\frac{|x-y|}{\varepsilon}\Big) f(y) dy \Big\}_\varepsilon
\Big\|_{\ell^\infty_{\mathbb{B}}}. $$ It is well known that the last
operator  can be viewed as  the $\ell^\infty_{\mathbb{B}}$-norm
 of a Calder\'on--Zygmund operator with regular kernel. Now we are in the situation of
 the proof of part $\textup{(ii)}\Rightarrow \textup{(i)}$
of Theorem C and with some obvious changes we get
$$\lim\limits_{\lambda\rightarrow
\infty}\left|\left\{x\in B_N:
\left|H^*_\varphi(f)(x)\right|>\lambda\right\}\right|=0, \quad
\forall f\in L^1_{\mathbb{B}}(\mathbb{R}),\ N>0.$$ In particular,
this implies that the operator $H^*_\varphi$  maps
$L^1_{\mathbb{B}}(\mathbb{R}) $ into $L^0(\mathbb{R})$. By
(\ref{equiv}) and the fact that $M$ maps
$L^1_{\mathbb{B}}(\mathbb{R})$ into weak-$L^1(\mathbb{R})$ for every
Banach space ${\mathbb{B}}$,  $H^*$ maps
$L^1_{\mathbb{B}}(\mathbb{R}) $ into $L^0(\mathbb{R})$. Now we can
apply the following lemma.

\begin{lem}\label{7.3}{\cite[Lemma~7.3]{MTX}}
Let $\mathbb{B}$ be a Banach space. Then every translation and
dilation invariant continuous sublinear operator $T:
L^1_\mathbb{B}(\mathbb{R}^n) \rightarrow L^0(\mathbb{R}^n) $ is of
weak type $(1,1)$.
\end{lem}
Then we get $H^*:L^1_{\mathbb{B}}(\mathbb{R}) \rightarrow
\hbox{weak-}L^1(\mathbb{R})$ which implies that the Banach space
${\mathbb{B}}$ is UMD.  This ends the proof of Theorem D.

\begin{rem} The above thoughts can be apply to the following general
situation.\\ Given two Banach spaces $\mathbb{B}_1$, $\mathbb{B}_2$
and $1\leq p<\infty,$ let $K(x,y) \in
\mathcal{L}(\mathbb{B}_1,\mathbb{B}_2)$ be a regular
Calder\'{o}n--Zygmund  kernel. Define $\displaystyle T_\varepsilon
f(x) = \int_{|x-y| >\varepsilon} K(x,y) f(y) dy$  and
$$Sf(x)=  \lim_{\varepsilon\rightarrow 0^+}T_\varepsilon f(x), \quad x\in \mathbb{R}^n.$$
Then the following statements are equivalent:
\begin{itemize}
\item For any $p\in (1,\infty),$ the operator $S$ maps $L^p_{\mathbb{B}_1}(\mathbb{R}^n)$ into $L^p_{\mathbb{B}_2}(\mathbb{R}^n). $
\item For any (or, equivalently, for some) $p\in (1, \infty),$ the maximal operator   $\displaystyle S^*f(x) =
\sup_{\varepsilon >0} \|T_{\varepsilon}f(x)\|_{\mathbb{B}_2} <
\infty, \, \hbox{ a.e. } x \in \mathbb{R}^n$ for every $f\in
L^p_{\mathbb{B}_1}(\mathbb{R}^n)$.
\end{itemize}

\noindent In other words, the following statement

 \centerline{``There exists a number
$p_0 \in [1, \infty) $ such that $\norm{Tf(x)}_{\mathbb{B}_2} <
\infty $ a.e $x\in \mathbb{R}^n,$ for every $f \in
L^{p_0}_{\mathbb{B}_1}(\mathbb{R}^n).$''} \noindent could be added
to the list of those statements in Remark \ref{CZ}, after an
appropriated description of $T$.
\end{rem}



\end{document}